\documentclass[a4paper,twoside,10pt]{amsart}
\usepackage[english]{babel}
\usepackage{t1enc}
\usepackage[charter]{mathdesign}
\usepackage{amsthm}
\usepackage{esint}
\newcommand*\Acal{\mathcal A}
\newcommand*\Ccal{\mathcal C}
\newcommand*\Mcal{\mathcal M}
\newcommand*\Pcal{\mathcal P}
\newcommand*\Qcal{\mathcal Q}
\newcommand*\Rcal{\mathcal R}
\newcommand*\Nbb{\mathbb N}
\newcommand*\Rbb{\mathbb R}
\newcommand*\nsubset{\not\subset}
\newcommand*\loc{\mathrm{loc}}
\newcommand*\fm{\mathrm{fin}}
\newcommand*\PI{\mathrm{PI}}
\newcommand*\dbl{\mathrm{dbl}}
\newcommand*\fcrim{\hspace{0.08333em}}
\newcommand*\emb{\hookrightarrow}
\newcommand*\dd{\mathrm{d}}
\newcommand*\cconc{c_{\scriptscriptstyle{\!\vartriangle}}}
\newcommand*\upto{\nearrow}
\newcommand*\NX{{N^1\!X}}
\newcommand*\eps{\varepsilon}
\newcommand*\Mod{\mathop{\mathrm{Mod}}\nolimits}
\newcommand*\diam{\mathop{\mathrm{diam}}\nolimits}
\newcommand*\dist{\mathop{\mathrm{dist}}\nolimits}
\newcommand*\spt{\mathop{\mathrm{spt}}\nolimits}
\newcommand*\Id{\mathop{\mathrm{Id}}\nolimits}
\newcommand*\essinf{\mathop{\mathrm{ess\,inf}}}
\newcommand*\ri{r.i.\@ }
\newcommand*{\coloneq}{\mathrel{\vcenter{\baselineskip0.5ex \lineskiplimit0pt\hbox{\scriptsize.}\hbox{\scriptsize.}}}=}
\newcommand*{\eqcolon}{=\mathrel{\vcenter{\baselineskip0.5ex \lineskiplimit0pt\hbox{\scriptsize.}\hbox{\scriptsize.}}}}
\newcommand*{\limplus}{{\mathchoice{\raise.17ex\hbox{$\scriptstyle +$}}
                {\raise.17ex\hbox{$\scriptstyle +$}}
                {\raise.1ex\hbox{$\scriptscriptstyle +$}}
                {\scriptscriptstyle +}}}
\newcommand*{\wMX}{M^*(X)}
\newcommand*{\meas}[1]{\mu \mathopen{}\left(#1\right)\mathclose{}}
\newcommand*{\measl}[1]{\mathopen{}\left|#1\right|\mathclose{}}
\newcommand*{\suplev}[2]{\mathcal{L}_{#1} (#2)}
\newcommand*{\Suplev}[2]{\mathcal{L}_{#1} \mathopen{}\left(#2\right)\mathclose{}}
\newcommand*{\suplevshp}[2]{\mathcal{L}^\sharp_{#1} (#2)}
\newcommand*{\suplevp}[3]{\mathcal{L}^{#1}_{#2} (#3)}
\newcommand*{\Suplevp}[3]{\mathcal{L}^{#1}_{#2} \mathopen{}\left(#3\right)\mathclose{}}
\newcommand*{\suplevr}[2]{\mathcal{J}_{#1} (#2)}
\renewcommand*{\theenumi}{(\alph{enumi})}

\newcommand*{\itoverline}[1]{\skew{3}{\overline}{#1}}
\newcommand*{\reps}[1]{\skew{3}{\overline}{#1}}
\theoremstyle{plain}
\newtheorem{thm}{Theorem}[section]
\newtheorem{lem}[thm]{Lemma}
\newtheorem{pro}[thm]{Proposition}
\newtheorem{cor}[thm]{Corollary}
\theoremstyle{definition}
\newtheorem{df}[thm]{Definition}
\newtheorem{exa}[thm]{Example}
\newtheorem{rem}[thm]{Remark}
\numberwithin{equation}{section}

\begin{document}
%
%
%
%
\begin{abstract}
Density of Lipschitz functions in Newtonian spaces based on quasi-Banach function lattices is discussed. Newtonian spaces are first-order Sobolev-type spaces on abstract metric measure spaces defined via (weak) upper gradients. Our main focus lies on metric spaces with a doubling measure that support a $p$-Poincar\'e inequality. Absolute continuity of the function lattice quasi-norm is shown to be crucial for approximability by (locally) Lipschitz functions. The proof of the density result uses, among others, that a suitable maximal operator is locally weakly bounded. In particular, various sufficient conditions for such boundedness on rearrangement-invariant spaces are established and applied.
\end{abstract}
%
%
%
%
\title[Regularization of Newtonian functions via weak boundedness of maximal operators]{Regularization of Newtonian functions on metric spaces\\via weak boundedness of maximal operators}
\author{Luk\'{a}\v{s} Mal\'{y}}
\date{April 28, 2014}
\address{Department of Mathematics\\Link\"{o}ping University\\SE-581~83~Link\"{o}ping\\Sweden}
\address{Department of Mathematical Analysis\\Faculty of Mathematics and Physics\\Charles University in Prague\\Sokolovsk\'a 83\\CZ-186~75~Praha 8\\Czech Republic}
\email{lukas.maly@liu.se}
\thanks{The author was partly supported by The Matts Ess\'en Memorial Fund.}
\maketitle{}
%
%
%
%
\section{Introduction}
\label{sec:intro}
{\setlength{\parskip}{0pt plus 0.5ex minus 0.2ex}
Newtonian functions represent an analogue and a generalization of first-order Sobolev functions in metric measure spaces. The notion of a distributional gradient relies heavily on the linear structure of $\Rbb^n$, which is missing in the setting of metric spaces. In the Newtonian theory, the distributional gradients are replaced by the so-called upper gradients or weak upper gradients, which were originally introduced by Heinonen and Koskela~\cite{HeiKos0} and Koskela and MacManus~\cite{KosMac}, respectively. The foundations for the Newtonian spaces $N^{1,p}$, based on the $L^p$ norm of a function and its (weak) upper gradient (i.e., corresponding to the classical Sobolev spaces $W^{1,p}$) were laid by Shanmugalingam~\cite{Sha}. In the past two decades, various authors have developed the elements of the Newtonian theory based on other function norms, see e.g.\@~\cite{CosMir,HarHasPer,Pod,Tuo}. Most recently, general complete quasi-normed lattices of measurable functions were considered as the base function space in Mal\'y~\cite{Mal1,Mal2}.

For many applications of the classical Sobolev spaces, it is of utmost importance that smooth functions are dense and provide good approximations of Sobolev functions. On metric spaces, the notion of a derivative, and hence of a smooth function, is unavailable; nevertheless, we may consider regularity in terms of (local) Lipschitz continuity. Such a regularity condition has turned out to suffice in many cases, e.g., within non-linear potential theory, see Bj\"{o}rn and Bj\"{o}rn~\cite{BjoBjo}. It has been shown already in Shanmugalingam's work~\cite{Sha} that Lipschitz functions are dense in $N^{1,p}(\Pcal)$ provided that $\Pcal$ is endowed with a doubling measure and supports a $p$-Poincar\'e inequality (see Definition~\ref{df:pPI} below). Tuominen~\cite{Tuo} has proven a similar result for Orlicz--Newtonian spaces with doubling Young function, while replacing the $p$-Poincar\'e inequality by an Orlicz-type Poincar\'e inequality.

Costea and Miranda~\cite{CosMir} studied the density of Lipschitz functions in Newtonian spaces based on the Lorentz $L^{p,q}$ spaces, assuming that $\Pcal$ carries an $L^{p,q}$-Poincar\'e inequality. They managed to prove the density for $1 \le q \le p<\infty$ using the fact that a Lorentz-type maximal operator is bounded from $L^{p,q}$ to $L^{p, \infty}$. They also found a counterexample for $1<p<q=\infty$. The case when $1\le p<q<\infty$ was however left open. Similar results were obtained earlier by Podbrdsk\'{y}~\cite{Pod} considering a more general setting of Banach space valued Lorentz functions, where the case $1\le p<q<\infty$ was not solved either.

It is known that Poincar\'e inequality is not a necessary condition to obtain the desired density. Using tools from optimal transportation theory, Ambrosio, Gigli and Savar\'e~\cite{AmbGigSav} argued that Lipschitz functions are dense in $N^{1,p}(\Pcal)$ for $p\in (1, \infty)$ if $\Pcal$ is compact and endowed with a doubling metric. Norm convergence of the sequence of approximating Lipschitz functions follows from reflexivity of $N^{1,p}(\Pcal)$, which was in that setting shown by Ambrosio, Colombo and Di~Marino~\cite{AmbColDiM}.

The current paper studies the question of density in situations when the base function space is a quasi-Banach function lattice with absolutely continuous quasi-norm. First, we provide a general theorem, where  Newtonian and Haj\l asz's theory of Sobolev-type spaces on metric measure spaces are intertwined. There,  we do not need to assume that $\Pcal$ carries any Poincar\'e inequality and the measure need not be doubling. The Haj\l asz gradient is however required to satisfy a weak type norm estimate. The connection between (weak) upper and Haj\l asz gradients is then established via the fractional sharp maximal operator and a $p$-Poincar\'e inequality. This leads to the assumption that a maximal operator of Hardy--Littlewood type (corresponding to the right-hand side of the $p$-Poincar\'e inequality supported by $\Pcal$) is weakly bounded on the function lattice. In particular, the open case in Lorentz--Newtonian spaces is settled with an affirmative answer. The presented results also extend the theory of Lipschitz truncations in variable exponent Newtonian spaces by Harjulehto, H\"ast\"o and Pere~\cite{HarHasPer} since we allow the infimum of the exponent to~be~$1$.

To determine whether Newtonian functions may be approximated by bounded functions is one of the steps towards the desired results. We will see that the absolute continuity of the function norm on sets of finite measure plays a vital role, which will help us with construction of examples where bounded functions are not dense in the quasi-Banach function lattice, whence neither are (locally) Lipschitz continuous functions in the corresponding Newtonian space.

One of the aims of the paper is to provide rather general theorems on the density of Lipschitz functions in Newtonian spaces with tangible hypotheses. Therefore, we also study when the suitable maximal operators are weakly bounded on sets of finite measure. We are particularly interested in their boundedness on rearrangement-invariant spaces and in its characterization in terms of the properties of the fundamental function.

We will prove that Lipschitz functions are dense in every Newtonian space based on a rearrangement-invariant space with absolutely continuous norm provided that $\Pcal$ supports a $1$-Poincar\'e inequality. If $\Pcal$ carries merely a $p$-Poincar\'e inequality with $p>1$, then it suffices, besides absolute continuity of the norm, that the upper fundamental (Zippin) or the upper Boyd index is less than $1/p$. Moreover, if $\Pcal$ is complete, then the indices may be equal to $1/p$. More generally, one can instead assume that $t \mapsto \phi(t)^p \fint_0^t \phi(s)^{-p}\,ds$ is bounded in a small neighborhood of $0$, where $\phi$ is the fundamental function of $X$.

If the Newtonian space is trivial, i.e., equal to the base function lattice, then the situation is much simpler. Regardless of the doubling condition of $\mu$, we give a general characterization of this triviality in terms of properties of the Sobolev capacity and of the $X$-modulus of a family of curves. In particular, we will see that the Newtonian space coincides with the base function lattice as sets if and only if their quasi-norms are equal. Such a characterization seems to be new even in the setting of the well-studied spaces $N^{1,p}$ that are built upon $L^p$. If a trivial Newtonian space is based on a Banach function space, then Lipschitz functions are dense whenever the norm is absolutely continuous.

The structure of the paper is the following. Section~\ref{sec:prelim} provides an overview of the used notation and preliminaries in the area of quasi-Banach function lattices and Newtonian spaces. Moreover, the characterization of triviality of a Newtonian space is given here. In Section~\ref{sec:trunc}, we study the density of truncated functions. Then, we obtain the general form of the main theorem about density of Lipschitz functions in Newtonian spaces in Section~\ref{sec:LipDensGeneral}, using the connection between Haj\l asz gradients, fractional sharp maximal operators and (weak) upper gradients. Rearrangement-invariant spaces lie in the focus of Section~\ref{sec:rispaces}. There, we also present a certain type of function spaces that will serve as counterexamples, where Newtonian functions cannot be approximated by Lipschitz functions. In Section~\ref{sec:weaktype}, we study maximal operators, with particular attention aimed at the weak boundedness in the setting of rearrangement-invariant spaces. Finally, Section~\ref{sec:LipDensSpec} contains various concretizations of the main result of the paper, giving  sufficient conditions for Lipschitz functions to be dense in the Newtonian space.
}
%
%
%
%
\section{Preliminaries}
\label{sec:prelim}
We assume throughout the paper that $\Pcal = (\Pcal, \dd, \mu)$ is a metric measure space equipped with a metric $\dd$ and a $\sigma$-finite Borel regular measure $\mu$ such that every ball in $\Pcal$ has finite positive measure. In our context, Borel regularity means that all Borel sets in $\Pcal$ are $\mu$-measurable and for each $\mu$-measurable set $A$ there is a Borel set $D\supset A$ such that $\meas{D} = \meas{A}$. Since $\mu$ is Borel regular and $\Pcal$ can be decomposed into countably many (possibly overlapping) open sets of finite measure, it is outer regular, see Mattila~\cite[Theorem 1.10]{Mat}.

The open ball centered at $x\in \Pcal$ with radius $r>0$ will be denoted by $B(x,r)$. Given a ball $B=B(x,r)$ and a scalar $\lambda > 0$, we let $\lambda B = B(x,\lambda r)$. We say that $\mu$ is a \emph{doubling} measure, if there is a constant $c_{\dbl}\ge1$ such that $\meas{2B} \le c_{\dbl} \meas{B}$ for every ball $B$. In Sections~\ref{sec:prelim} and~\ref{sec:trunc}, unlike in the rest of the paper, we will not assume that $\mu$ is doubling or non-atomic.

Let $\Mcal(\Pcal, \mu)$ denote the set of all extended real-valued $\mu$-measurable functions on $\Pcal$. The set of extended real numbers, $\Rbb \cup \{\pm \infty\}$, will be denoted by $\overline{\Rbb}$. We will also use $\Rbb^+$, which denotes the set of positive real numbers, i.e., the interval $(0, \infty)$. The symbol $\Nbb$ will denote the set of positive integers, i.e., $\{1,2, \ldots\}$. We define the \emph{integral mean} of a measurable function $u$ over a set $E$ of finite positive measure as
\[
  u_E \coloneq \fint_E u\,d\mu = \frac{1}{\mu(E)} \int_E u\,d\mu,
\]
whenever the integral on the right-hand side exists, not necessarily finite though. The characteristic function of a set $E$ will be denoted by $\chi_E$. Given an extended real-valued function $u: \Pcal \to \overline{\Rbb}$ and a real number $\sigma\ge 0$, we define $\suplev{u}{\sigma}$ as the superlevel set $\{x\in \Pcal: |u(x)| > \sigma\}$.
The notation $L \lesssim R$ will be used to express that there exists a constant $c>0$, perhaps dependent on other constants within the context, such that $ L \le cR$. If $L \lesssim R$ and simultaneously $R \lesssim L$, then we will simply write $L \approx R$ and say that the quantities $L$ and $R$ are \emph{comparable}. The words \emph{increasing} and \emph{decreasing} will be used in their non-strict sense.

A linear space $X = X(\Pcal, \mu)$ of equivalence classes of functions in $\Mcal(\Pcal, \mu)$ is said to be a \emph{quasi-Banach function lattice} over $(\Pcal, \mu)$ equipped with the quasi-norm $\|\cdot\|_X$ if the following axioms hold:
\begin{enumerate}
  \renewcommand{\theenumi}{(P\arabic{enumi})}
  \setcounter{enumi}{-1}
  \item \label{df:qBFL.initial} $\|\cdot\|_X$ determines the set $X$, i.e., $X = \{u\in \Mcal(\Pcal, \mu)\colon \|u\|_X < \infty\}$;
  \item \label{df:qBFL.quasinorm} $\|\cdot\|_X$ is a \emph{quasi-norm}, i.e., 
  \begin{itemize}
    \item $\|u\|_X = 0$ if and only if $u=0$ a.e.,
    \item $\|au\|_X = |a|\,\|u\|_X$ for every $a\in\Rbb$ and $u\in\Mcal(\Pcal, \mu)$,
    \item there is a constant $\cconc \ge 1$, the so-called \emph{modulus of concavity}, such that $\|u+v\|_X \le \cconc(\|u\|_X+\|v\|_X)$ for all $u,v \in \Mcal(\Pcal, \mu)$;
  \end{itemize}
  \item \label{df:BFL.latticeprop} $\|\cdot\|_X$ satisfies the \emph{lattice property}, i.e., if $|u|\le|v|$ a.e., then $\|u\|_X\le\|v\|_X$;
  \renewcommand{\theenumi}{(RF)}
  \item \label{df:qBFL.RF} $\|\cdot\|_X$ satisfies the \emph{Riesz--Fischer property}, i.e., if $u_n\ge 0$ a.e.\@ for all $n\in\Nbb$, then $\bigl\|\sum_{n=1}^\infty u_n \bigr\|_X \le \sum_{n=1}^\infty \cconc^n \|u_n\|_X$, where $\cconc\ge 1$ is the modulus of concavity. Note that the function $\sum_{n=1}^\infty u_n$ needs to be understood as a pointwise (a.e.\@) sum.
\end{enumerate}
Observe that $X$ contains only functions that are finite a.e., which follows from~\ref{df:qBFL.initial}--\ref{df:BFL.latticeprop}. In other words, if $\|u\|_X<\infty$, then $|u|<\infty$ a.e.

Throughout the paper, we will also assume that the quasi-norm $\|\cdot\|_X$ is \emph{continuous}, i.e., if $\|u_n - u\|_X \to 0$ as $n\to\infty$, then $\|u_n\|_X \to \|u\|_X$. We do not lose any generality by this assumption as the Aoki--Rolewicz theorem (see Benyamini and Lindenstrauss~\cite[Proposition H.2]{BenLin} or Maligranda~\cite[Theorem~1.2]{Mali}) implies that there is always an equivalent quasi-norm that is an $r$-norm, i.e., it satisfies
\[
  \|u + v\|^r \le \|u\|^r + \|v\|^r,
\]
where $r = 1/(1+ \log_2 \cconc) \in (0, 1]$, which implies the continuity. The theorem's proof shows that such an equivalent quasi-norm retains the lattice property~\ref{df:BFL.latticeprop}.

It is worth noting that the Riesz--Fischer property is actually equivalent to the completeness of the quasi-normed space $X$, given that the conditions \ref{df:qBFL.initial}--\ref{df:BFL.latticeprop} are satisfied and that the quasi-norm is continuous, see Maligranda~\cite[Theorem 1.1]{Mali}.

If $\cconc = 1$, then the functional $\| \cdot \|_X$ is a norm. We then drop the prefix \emph{quasi} and hence call $X$ a \emph{Banach function lattice}.

A (quasi)Banach function lattice $X = X(\Pcal, \mu)$ is called a \emph{(quasi)Banach function space} over $(\Pcal, \mu)$ if the following axioms are satisfied as well:
\begin{enumerate}
  \renewcommand{\theenumi}{(P\arabic{enumi})}
  \setcounter{enumi}{2}
  \item $\|\cdot\|_X$ satisfies the \emph{Fatou property}, i.e., if $0\le u_n \upto u$ a.e., then $\|u_n\|_X\upto\|u\|_X$;
  \item \label{df:BFS.finmeasfinnorm} if a measurable set $E \subset \Pcal$ has finite measure, then $\|\chi_E\|_X < \infty$;
  {
  \item for every measurable set $E\subset \Pcal$ with $\meas{E}<\infty$ there is $C_E>0$ such that $\int_E |u|\,d\mu \le C_E \|u\|_X$ for every measurable function $u$.
    \label{df:BFL.locL1}
  } 
\end{enumerate}
Note that the Fatou property implies the Riesz--Fischer property. Axiom~\ref{df:BFS.finmeasfinnorm} is equivalent to the condition that $X$ contains all simple functions (with support of finite measure). Due to the lattice property~\ref{df:BFL.latticeprop}, we can equivalently characterize~\ref{df:BFS.finmeasfinnorm} as embedding of $L^\infty(\Pcal, \mu)$ into $X$ on sets of finite measure. Finally, condition~\ref{df:BFL.locL1} describes that $X$ is embedded into $L^1(\Pcal, \mu)$ on sets of finite measure.

In the further text, we will slightly deviate from this rather usual definition of (quasi)\allowbreak{}Banach function lattices and spaces. Namely, we will consider $X$ to be a linear space of functions defined everywhere instead of equivalence classes defined a.e. Then, the functional $\|\cdot\|_X$ is really only a (quasi)seminorm. Unless explicitly stated otherwise, we will always assume that $X$ is a quasi-Banach function lattice.

The quasi-norm $\| \cdot \|_X$ in a quasi-Banach function lattice $X$ is \emph{absolutely continuous} if every $u \in X$ satisfies the condition
\begin{enumerate}
  \renewcommand{\theenumi}{(AC)}
  \item \label{df:AC}
  $\| u \chi_{E_n} \|_X \to 0$ as $n\to\infty$ whenever $\{E_n\}_{n=1}^\infty$ is a decreasing sequence of measurable sets with $\meas{\bigcap_{n=1}^\infty E_n} = 0$.
\end{enumerate}
It follows from the dominated convergence theorem that the $L^p$ norm is absolutely continuous for $p\in (0, \infty)$. On the other hand, $L^\infty$ lacks this property apart from in a few exceptional cases. For example, if $\mu$ is atomic, $0<\delta\le \meas{A}$ for every atom $A\subset \Pcal$, and $\meas{\Pcal}<\infty$, then every quasi-Banach function lattice has absolutely continuous quasi-norm since the condition $\meas{\bigcap_{n=1}^\infty E_n} = 0$ implies that there is $n_0 \in \Nbb$ such that $E_n=\emptyset$ for all $n\ge n_0$. However, atomic measures lie outside of the main scope of our interest.
%
%
\begin{df}
For $p \in [1, \infty)$, the \emph{non-centered maximal operator} $M_p$ is defined by
\[
  M_p u(x) = \sup_{B\ni x} \biggl(\fint_B |u|^p\,d\nu\biggr)^{1/p}, \quad x\in \Rcal,
\]
where $(\Rcal, \nu)$ is a given metric measure space and $u \in \Mcal(\Rcal, \nu)$.
We also define the superlevel set
\[
  \suplevp{p}{u}{\sigma} \coloneq \suplev{M_pu}{\sigma} = \{ x\in \Rcal: M_p u(x) > \sigma \} \quad \mbox{for }\sigma \ge 0.
\]
\end{df}
%
%
\begin{rem}
We will use either $(\Rcal, \nu) = (\Rbb^+, \lambda^1)$ or $(\Rcal, \nu) = (\Pcal, \mu)$ depending on the context, yet without any explicit indication of which of the cases applies at the moment. Since $M_p u= (M_1 |u|^p)^{1/p}$, we obtain that $M_p: L^p \to L^{p,\infty}$ is bounded due to the weak-$L^1$ boundedness of $M_1$ on doubling spaces (see e.g.\@ Coifman and Weiss~\cite[Theorem III.2.1]{CoiWei}). Obviously, $M_p: L^\infty \to L^\infty$ is also bounded.
\end{rem}
%
%
Given a function lattice $X$, we also define a ``local'' space $X_\fm$ that consists of measurable functions whose restrictions to sets of finite measure belong to $X$, i.e., $u\in X_\fm$ if $u \chi_E \in X$ for every measurable set $E$ with $\meas{E}<\infty$. If $\meas{\Pcal} < \infty$, then obviously $X_\fm = X$.
We say that a (sub)linear mapping $T: X(\Pcal, \mu) \to Y_\fm(\Rcal, \nu)$ is bounded, if for every $E\subset \Rcal$ with $\nu(E)<\infty$ there is $c_E>0$ such that $\|(Tu) \chi_E\|_Y \le c_E \|u\|_X$ whenever $u\in X$. It might actually happen that $Tu \notin Y$ even though $u\in X$. If $\nu(\Rcal)<\infty$, then $T: X \to Y_\fm$ is bounded if and only if $T: X\to Y$ is bounded. We will also say that $X$ is \emph{continuously embedded} in $Y_\fm$, which will be denoted by $X \emb Y_\fm$, if the identity mapping $\Id: X(\Pcal, \mu) \to Y_\fm(\Pcal, \mu)$ is bounded.

By a \emph{curve} in $\Pcal$ we will mean a non-constant continuous mapping $\gamma: I\to \Pcal$ with finite total variation (i.e., length of $\gamma(I)$), where $I \subset \Rbb$ is a compact interval. Thus, a curve can be (and we will always assume that all curves are) parametrized by arc length $ds$, see e.g.\@ Heinonen~\cite[Section~7.1]{Hei}. Note that every curve is Lipschitz continuous with respect to its arc length parametrization. The family of all non-constant rectifiable curves in $\Pcal$ will be denoted by $\Gamma(\Pcal)$. By abuse of notation, the image of a curve $\gamma$ will also be denoted by $\gamma$. 

A statement holds for \emph{$\Mod_X$-a.e.\@} curve $\gamma$ if the family of exceptional curves $\Gamma_e$, for which the statement fails, has \emph{zero $X$-modulus}, i.e., if there is a Borel function $\rho \in X$ such that $\int_\gamma \rho\,ds = \infty$ for every curve $\gamma \in \Gamma_e$ (see Mal\'{y}~\cite[Proposition 4.8]{Mal1}).
%
%
\begin{df}
\label{df:ug}
  Let $u: \Pcal \to \overline{\Rbb}$. Then, a Borel function $g: \Pcal \to [0, \infty]$ is an \emph{upper gradient} of $u$ if
\begin{equation}
 \label{eq:ug_def}
 |u(\gamma(0)) - u(\gamma(l_\gamma))| \le \int_\gamma g\,ds
\end{equation}
for every curve $\gamma: [0, l_\gamma]\to\Pcal$. To make the notation easier, we are using the convention that $|(\pm\infty)-(\pm\infty)|=\infty$. If we allow $g$ to be a measurable function and \eqref{eq:ug_def} to hold only for $\Mod_X$-a.e.\@ curve $\gamma: [0, l_\gamma]\to\Pcal$, then $g$ is an \emph{$X$-weak upper gradient}.
\end{df}
%
%
Observe that the ($X$-weak) upper gradients are by no means given uniquely. Indeed, if we have a function $u$ with an ($X$-weak) upper gradient $g$, then $g+h$ is another ($X$-weak) upper gradient of $u$ whenever $h\ge0$ is a Borel (measurable) function.
%
%
\begin{df}
Whenever $u\in \Mcal(\Pcal, \mu)$, let
\begin{equation}
  \label{eq:def-N1X-norm}
  \|u\|_{\NX} = \| u \|_X + \inf_g \|g\|_X,
\end{equation}
where the infimum is taken over all upper gradients $g$ of $u$. The \emph{Newtonian space} based on $X$ is the space
\[
  \NX = \NX (\Pcal, \mu)= \{u\in\Mcal(\Pcal, \mu): \|u\|_{\NX} <\infty \}.
\]
\end{df}
Let us point out that we assume that functions are defined everywhere, and not just up to equivalence classes $\mu$-almost everywhere. This is essential for the notion of upper gradients since they are defined by a pointwise inequality.

The functional $\| \cdot \|_\NX$ is a quasi-seminorm on $\NX$ and its modulus of concavity equals the modulus $\cconc$ of the base function space $X$. We may very well take the infimum over all $X$-weak upper gradients $g$ of $u$ in \eqref{eq:def-N1X-norm} without changing the value of the Newtonian quasi-seminorm. Moreover, $\NX$ is complete (see~\cite[Theorem~7.1]{Mal1}).

The \emph{(Sobolev) capacity}, defined as $C_X(E) = \inf\{\|u\|_\NX: u\ge \chi_E\}$ for $E\subset \Pcal$, is a set function that distinguishes which sets do not carry any information about a Newtonian function and thus are negligible. The natural equivalence classes in $\NX$ are given by equality outside of sets of zero capacity. These as well as other basic properties of Newtonian functions have been established in~\cite{Mal1}.

It has been shown in~\cite{Mal2} that the infimum in \eqref{eq:def-N1X-norm} is attained for functions in $\NX$ by a \emph{minimal $X$-weak upper gradient}. Such an $X$-weak upper gradient is minimal both normwise and pointwise (a.e.\@) among all ($X$-weak) upper gradients in $X$, whence it is given uniquely up to equality a.e.

The following lemma provides us with several equivalent conditions that describe triviality of the Newtonian space in the sense that $\NX = X$. Such a characterization seems to be new even for the well-studied spaces $N^{1,p}\coloneq N^1L^p$.
%
%
\begin{lem}
\label{lem:NX=X}
The following are equivalent:
\begin{enumerate}
	\item \label{it:NX=X-sets} $\NX = X$ as sets of functions;
	\item \label{it:NX=X-cap} $C_X(E) = 0$ if and only if $\mu(E) = 0$, where $E\subset \Pcal$;
	\item \label{it:NX=X-mod} $\Mod_X(\Gamma(\Pcal)) = 0$;
	\item \label{it:NX=X-norms} $\|u\|_\NX = \|u\|_X$ for every $u\in\Mcal(\Pcal, \mu)$.
\end{enumerate}
\end{lem}
\begin{proof}
\ref{it:NX=X-sets} $\Rightarrow$~\ref{it:NX=X-cap} Let $E\subset \Pcal$ satisfy $\mu(E) = 0$. Then, $u = \infty \chi_E$ belongs to $X$ since $\|u\|_X = \|0\|_X = 0$. Hence, $u\in \NX$. Thus, $C_X(E) = C_X(\{x\in \Pcal: |u(x)|=\infty\}) = 0$ by~\cite[Proposition 3.6]{Mal1}.

\ref{it:NX=X-cap} $\Rightarrow$~\ref{it:NX=X-mod} Let $\{x_n \in \Pcal: n\in\Nbb\}$ be a dense subset of $\Pcal$. For each $n\in\Nbb$, we can find a set of radii $\{r_{n,k}>0: k\in\Nbb\}$, dense in $(0, \infty)$, such that the spheres $S(x_n, r_{n,k}) = \{z\in \Pcal: \dd(x_n, z) = r_{n,k}\}$ satisfy $\meas{S(x_n, r_{n,k})} = 0$ for every $k\in\Nbb$. Such a sequence indeed exists since at most countably many spheres centered at $x_n$ have positive measure as balls would not have finite measure otherwise.

Let $E_n = \bigcup_{k=1}^\infty S_{n,k}$. Then, $\meas{E_n} = 0 = C_X(E_n)$. Therefore, $\Mod_X(\Gamma_{E_n}) = 0$ by~\cite[Proposition~5.10]{Mal1}, where $\Gamma_{E_n} = \{\gamma\in\Gamma(\Pcal): \gamma^{-1}(E_n) \neq \emptyset\}$. Let
\[ 
  \Gamma_n = \{ \gamma \in \Gamma(\Pcal): \dd(x_n, \gamma(t_1)) \neq \dd(x_n, \gamma(t_2))\mbox{ for some }0\le t_1 < t_2 \le l_\gamma \}.
\]
Then, $\Gamma_n \subset \Gamma_{E_n}$ and hence $\Mod_X(\Gamma_n) \le \Mod_X(\Gamma_{E_n}) = 0$. As there are no (non-constant) curves that have a constant distance from all points $x_n$, $n\in\Nbb$, we obtain that
\[
  \Mod_X(\Gamma(\Pcal)) = \Mod_X \biggl(\bigcup_{n=1}^\infty \Gamma_n\biggr) = 0. 
\]

\ref{it:NX=X-mod} $\Rightarrow$~\ref{it:NX=X-norms} Since \eqref{eq:ug_def} is allowed to fail for  every curve $\gamma \in \Gamma(\Pcal)$, $g \equiv 0$ is an \mbox{$X$-}weak upper gradient of every measurable function $u\in\Mcal(\Pcal, \mu)$, whence $\|u\|_X \le \|u\|_\NX \le \|u\|_X + \|0\|_X = \|u\|_X$.

\ref{it:NX=X-norms} $\Rightarrow$~\ref{it:NX=X-sets} If the quasi-norms are equal, then $X = \{ u\in \Mcal(\Pcal, \mu): \|u\|_X < \infty\} = \{ u\in X: \|u\|_\NX < \infty\} = \NX$.
\end{proof}
%
%
In the next proposition, we demonstrate that the density of Lipschitz functions relies only on the properties of $X$ whenever the Newtonian space is trivial.
%
%
\begin{pro}
\label{pro:NX=X-dens}
Let $X$ be a Banach function space with absolutely continuous norm, i.e., it satisfies~\ref{df:qBFL.initial}--\ref{df:BFL.locL1} and~\ref{df:AC}. Suppose that $\NX = X$. Then, Lipschitz functions are dense in $\NX$.
\end{pro}
%
%
\begin{proof}
Simple functions are dense in $X$ by~\cite[Theorem~I.3.11]{BenSha}.
Let $E\subset \Pcal$ be a measurable set of finite measure and $\eps>0$ be arbitrary. Then, there exists a bounded set $E_b \subset E$ such that $\meas{E\setminus E_b} < \eps$. Let $\itoverline{B}\subset \Pcal$ be a closed ball that contains $E_b$. By outer regularity of $\mu$, there is an open set $G \supset \itoverline{B} \setminus E_b$ such that $\meas{G \cap E_b} \le \meas{G \setminus (\itoverline{B} \setminus E_b)} < \eps$. Let $F = \itoverline{B} \setminus G$. Then, $F$ is closed in $\itoverline{B}$ and hence in $\Pcal$, and $\meas{E_b \setminus F} = \meas{E_b \cap G}< \eps$. Thus, $\meas{E \setminus F} < 2\eps$.

Therefore, for every measurable $E \subset \Pcal$ of finite measure, there is a bounded closed set $F \subset E$ such that $\|\chi_{E\setminus F}\|_X$ is arbitrarily small by the absolute continuity of the norm. For such a set $F$, we define $\eta_k(x) = (1- k \dist(x, F))^+$, $x\in \Pcal$, $k\in\Nbb$. Then, $\eta_k$ has bounded support, whence $\eta_k \in X$. The function $\eta_k$ is $k$-Lipschitz, and $\eta_k \to \chi_F$ a.e.\@ in $\Pcal$ as $k\to\infty$. By the dominated convergence theorem (which follows from the absolute continuity, see~\cite[Proposition I.3.6]{BenSha}), we obtain $\eta_k \to \chi_F$ in $X$ as $k\to\infty$. Therefore, every simple function can be approximated in the norm of $X$ by Lipschitz functions.

Consequently, every $u\in X=\NX$ can be approximated in the norm of $\NX$ by Lipschitz functions since the norms of $X$ and $\NX$ are equal by Lemma~\ref{lem:NX=X}.
\end{proof}
%
%
\begin{df}
\label{df:pPI}
We say that $\Pcal$ supports a \emph{$p$-Poincar\'{e} inequality} or, for the sake of brevity, that $\Pcal$ is a \emph{$p$-Poincar\'{e} space} if there exist constants $c_{\PI} > 0$ and $\lambda \ge 1$ such that for all balls $B\subset \Pcal$, for all $u\in L^1_\loc(\Pcal)$ and all upper gradients $g$ of $u$,
\begin{equation}
  \label{eq:pPI}
  \fint_B |u-u_B|\,d\mu \le c_{\PI} \diam (B) \Biggl( \fint_{\lambda B} g^p\,d\mu\Biggr)^{1/p},
\end{equation}
where $u_B = \fint_B u\,d\mu$.
\end{df}
This form of the inequality is sometimes called a \emph{weak $p$-Poincar\'{e} inequality}. The word ``weak'' indicates that the dilation factor $\lambda$ is allowed to be greater than $1$. Note also that it follows by~\cite[Lemma 5.6]{Mal1} that we may equivalently require that the inequality holds for all $p$-weak upper gradients $g$ of $u$ and, in particular, for all $X$-weak upper gradients $g$ of $u$ if $X \emb L^p_\loc$, i.e., if $\|f \chi_B\|_{L^p} \le c_B \|f\chi_B\|_X$ for all balls $B\subset \Pcal$. There are several other characterizations in~\cite[Proposition 4.13]{BjoBjo}, e.g., we may require that \eqref{eq:pPI} holds only for $u \in L^\infty$, or conversely that it holds for all measurable functions $u$ if the left-hand side is interpreted as $\infty$ whenever $u\chi_B \notin L^1$.
%
%
%
%
\section{Approximation by bounded functions}
\label{sec:trunc}
In this section, we will determine a set of sufficient conditions ensuring that truncated functions  provide a good approximation of Newtonian functions, which is an important step on the way to study the density of Lipschitz functions as these are bounded on bounded sets. In Section~\ref{sec:rispaces}, we will find a certain type of function spaces where the truncations are not dense, which will lead us later on to constructing examples when (locally) Lipschitz functions are not dense in the Newtonian space.
%
%
\begin{lem}
\label{lem:trunc_dense_AC}
Let $X$ be a quasi-Banach function lattice with absolutely continuous quasi-norm. Then every function $u \in X$ can be approximated by its truncations with arbitrary precision in the norm of $X$, i.e., if we define $u_\sigma \coloneq \max \{ \min \{u, \sigma\}, -\sigma\}$ for $\sigma\in\Rbb^+$, then $u_\sigma \to u$ in $X$ as $\sigma\to\infty$.
\end{lem}
%
%
Recall that $\suplev{u}{\sigma}$ denotes the superlevel set of $|u|$ with level $\sigma \ge 0$, i.e., $\suplev{u}{\sigma} = \{x\in \Pcal: |u(x)| > \sigma\}$.
%
%
\begin{proof}
Let $u\in X$ and let $u_\sigma$ be its truncations at the levels $\pm \sigma$ for every $\sigma \in \Rbb^+$. Then, $u-u_\sigma = 0$ on $\Pcal \setminus \suplev{u}{\sigma}$. Since $|u|<\infty$ a.e.\@ in $\Pcal$, we have that $\meas{\bigcap_{\sigma>0} \suplev{u}{\sigma}} = 0$. The absolute continuity of the quasi-norm of $X$ implies that
\[
  \| u - u_\sigma\|_X =\| (|u|-\sigma) \chi_{\suplev{u}{\sigma}}\|_X \le \| u \chi_{\suplev{u}{\sigma}} \|_X \to 0\quad\mbox{as }\sigma\to\infty.
\qedhere
\]
\end{proof}
%
%
The following lemma shows that the measure of the superlevel sets of an $L^p_\fm$ function is finite if the level is chosen sufficiently large. In fact, it tends to zero as the level approaches infinity.
\begin{lem}
  \label{lem:superlevelsets-to-zero}
  Let $u \in L^p_\fm\fcrim$ for some $p>0$. Suppose further that $\mu$ is non-atomic. Then, $\mu(\suplev{u}{\sigma}) \to 0$ as $\sigma \to \infty$.
\end{lem}
\begin{proof}
Since $|u|<\infty$ a.e., we obtain that $\meas{\bigcap_{\sigma>0} \suplev{u}{\sigma}} = 0$. If we show that $\meas{\suplev{u}{\sigma}} < \infty$ for some $\sigma>0$, then $\meas{\bigcap_{\sigma>0} \suplev{u}{\sigma}} = \lim_{\sigma \to \infty} \meas{\suplev{u}{\sigma}}$. 

Suppose on the contrary that $\meas{\suplev{u}{\sigma}} = \infty$ for every $\sigma>0$. Then, we can construct a set $F$ of finite measure such that $u\chi_F\notin L^p$ as follows. Let us choose a sequence of pairwise disjoint sets $F_k$, where $\meas{F_k} = 1/k^2$ and $F_k \subset \suplev{u}{k^{1/p}}$. Let now $F = \bigcup_{k=1}^\infty F_k$. Then, $\meas{F}<\infty$, but
\[
  \| u \chi_F \|_{L^p}^p \ge \sum_{k=1}^\infty k \meas{F_k} = \sum_{k=1}^\infty \frac{1}{k} = \infty,
\]
whence $u\chi_F \notin L^p$, which contradicts the assumption that $u \in L^p_\fm$.
\end{proof}
%
%
In order to investigate whether truncated functions are good approximations in Newtonian spaces, we need to check how truncation affects weak upper gradients. The following auxiliary lemma will help us settle this problem as the gradient may be modified so that it vanishes on a given level set of a Newtonian function.
%
%
\begin{lem}
\label{lem:glueing_lemma}
Let $X$ be a quasi-Banach function lattice. Suppose that $u\in\NX$ with an $X$-weak upper gradient $g\in X$. Given a constant $k\in\Rbb$, define $E=\{x\in\Pcal: u(x) = k\}$. Then, $g \chi_{\Pcal \setminus E}$ is an $X$-weak upper gradient of $u$ as well.
\end{lem}
%
%
\begin{proof}
Let $\tilde{g}$ be a Borel representative of $g$. We will show that $\tilde{g} \chi_{\Pcal\setminus E}$ is an $X$-weak upper gradient of $u$ and hence so is $g \chi_{\Pcal\setminus E}$ by~\cite[Lemma~4.10]{Mal1}. For $\Mod_X$-a.e.\@ curve~$\gamma$ we have that $u$ is absolutely continuous on $\gamma$ by~\cite[Theorem 6.7]{Mal1} and $\tilde{g}$ satisfies \eqref{eq:ug_def} for every subcurve $\gamma' = \gamma|_I$ by~\cite[Corollary 5.9]{Mal1}, where $I\subset [0, l_\gamma]$ is a closed interval. Let $\gamma: [0, l_\gamma] \to \Pcal$ be such a curve. If $\gamma \cap E = \emptyset$, then $\tilde{g}=\tilde{g} \chi_{\Pcal \setminus E}$ everywhere on $\gamma$. Suppose now that the curve $\gamma$ intersects with the set $E$. Let
\[
  \alpha = \inf\{ t\in[0, l_\gamma]: \gamma(t) \in E\}, \quad\mbox{and}\quad \beta = \sup\{ t\in[0, l_\gamma]: \gamma(t) \in E\}.
\]
Hence, $\gamma([0, \alpha)) \cap E = \emptyset = \gamma((\beta, l_\gamma]) \cap E $ and $\tilde{g}\circ \gamma = (\tilde{g}\chi_{\Pcal \setminus E})\circ \gamma$ on $[0, \alpha)\cup(\beta, l_\gamma]$. Furthermore, $u(\gamma(\alpha)) = u(\gamma(\beta)) = k$ since $u \circ \gamma \in \Ccal([0, l_\gamma])$. Consequently,
\begin{align*}
  |u(\gamma(0)) - u(\gamma(\alpha))| &\le \int_0^\alpha \tilde{g}(\gamma(t))\,dt = \int_0^\alpha (\tilde{g}\chi_{\Pcal\setminus E})(\gamma(t))\,dt,\\
  |u(\gamma(\alpha)) - u(\gamma(\beta))| & = 0 \le \int_\alpha^\beta (\tilde{g}\chi_{\Pcal\setminus E})(\gamma(t))\,dt,\\
  |u(\gamma(\beta)) - u(\gamma(l_\gamma))| &\le \int_\beta^{l_\gamma} \tilde{g}(\gamma(t))\,dt = \int_\beta^{l_\gamma} (\tilde{g}\chi_{\Pcal\setminus E})(\gamma(t))\,dt.
\end{align*}
These estimates together give that $|u(\gamma(0)) - u(\gamma(l_\gamma))| \le \int_\gamma \tilde{g}\chi_{\Pcal\setminus E}\,ds$ holds for $\Mod_X$-a.e.\@ curve $\gamma$ whence $\tilde{g}\chi_{\Pcal\setminus E}$ is an $X$-weak upper gradient of $u$ and so is $g\chi_{\Pcal\setminus E}$.
\end{proof}
%
%
Now, we are ready to prove that truncated functions are dense in $\NX$ as well, provided that $X$ has absolutely continuous quasi-norm. In Example~\ref{exa:trunc_not_dense} below, the absolute continuity is shown to be crucial for the density of truncations in $\NX$.
%
%
\begin{cor}
\label{cor:bdd-dense-in-N1X}
Let $X$ be a quasi-Banach function lattice with absolutely continuous quasi-norm. Then, every function $u \in \NX$ can be approximated by its truncations with arbitrary precision in $\NX$, i.e., if $u_\sigma \coloneq \max \{ \min \{u, \sigma\}, -\sigma\}$ for $\sigma\in\Rbb^+$, then $u_\sigma \to u$ in $\NX$ as $\sigma\to\infty$.
\end{cor}
%
%
\begin{proof}
Let $u\in\NX$ be given and suppose that $g_u \in X$ is its minimal $X$-weak upper gradient. Then, $g_u$ is an $X$-weak upper gradient of $u-u_\sigma$ as well. The previous lemma implies that $g_u\chi_{\suplev{u}{\sigma}}$ is also an $X$-weak upper gradient of $u-u_\sigma$ as
\[
  \suplev{u}{\sigma} = \{x\in\Pcal: |u(x)|>\sigma \} = \{x\in\Pcal: (u-u_\sigma)(x) \neq 0\}.
\]
Since $\meas{\bigcap_{\sigma>0} \suplev{u}{\sigma}} = 0$, the absolute continuity of the norm of $X$ leads to
\[
  \|u-u_\sigma\|_\NX \le \|u \chi_{\suplev{u}{\sigma}}\|_X + \|g_u \chi_{\suplev{u}{\sigma}}\|_X \to 0\quad \mbox{as }\sigma\to\infty.
  \qedhere
\]
\end{proof}
%
%
%
%
\section{Main results in their general form}
\label{sec:LipDensGeneral}
The main results on density of Lipschitz functions in Newtonian spaces are stated and proven in this section. Here, we show general theorems and provide examples of Newtonian spaces where they can be readily applied. Various special cases of the main theorems, whose hypotheses are easier to verify, will be discussed in Section~\ref{sec:LipDensSpec}. Recall that in this as well as in all subsequent sections, we will assume that $\mu$ is a non-atomic doubling measure (unless explicitly stated otherwise).

A different approach to study Sobolev-type functions on metric measure spaces was proposed by Haj\l{}asz in~\cite{Haj96}. Instead of (weak) upper gradients, another type of gradient was used, which allows a simple construction of Lipschitz approximations.
\begin{df}
\label{df:haj-grad}
Let $u: \Pcal \to \overline{\Rbb}$. Then, a measurable function $h: \Pcal \to [0, \infty]$ is a \emph{Haj\l{}asz gradient} of $u$ if there is a set $E \subset \Pcal$ with $\meas{E} = 0$ such that
\begin{equation}
  \label{eq:def-haj}
  |u(x) - u(y) | \le \dd(x,y) (h(x) + h(y)) \quad \mbox{for every $x,y \in \Pcal \setminus E$.}
\end{equation}
\end{df}
%
%
Jiang, Shanmugalingam, Yang, and Yuan~\cite[Theorem~1.3]{JiaShaYanYua} have shown that $4h$ is an $X$-weak upper gradient of a suitable representative of a function $u\in X \subset L^1_\loc$ with a Haj\l{}asz gradient $h \in X$, provided that $\mu$ is doubling. The main idea of this claim can be traced back to J.~Mal\'y, cf.~Haj\l{}asz~\cite[Proposition 1]{Haj94}. Slightly improved version can be found in Heinonen, Koskela, Shanmugalingam and Tyson~\cite[Lemma~9.2.5]{HeiKosShaTys}, where $3h$ is shown to be an upper gradient of $u\in \Ccal\cap L^1_\loc$ with a Haj\l{}asz gradient $h\in L^1_\loc$. Such a result is further refined in~\cite{Mal5}, where $2h$ is proven to be an $X$-weak upper gradient of a measurable function $u$ that is absolutely continuous on $\Mod_X$-a.e.\@ curve, regardless of the doubling condition of $\mu$ and regardless of the summability of $u$ or $h$. Moreover, the factor $2$ is shown to be optimal.

On the other hand, without any additional assumptions on the metric measure space, it is in general impossible to find a Haj\l{}asz gradient using a (weak) upper gradient of a function. If $\Pcal$ supports a Poincar\'e inequality, then a certain maximal function of an upper gradient is a Haj\l{}asz gradient, see the proof of Theorem~\ref{thm:general_Lip-dens} below or Haj\l{}asz~\cite{Haj}, cf.\@ also Shanmugalingam~\cite[Theorem~4.9]{Sha}.
%
%
\begin{thm}
\label{thm:Haj_Lip-dens}
Let $X$ be a quasi-Banach function lattice with absolutely continuous quasi-norm. Suppose that $u\in \NX$ has a Haj\l{}asz gradient $h$ that satisfies the weak estimate $\|\sigma \chi_{\suplev{h}{\sigma}} \|_X \to 0$ as $\sigma \to \infty$. (In particular, it suffices to suppose that $h\in X$.) Then, for every $\eps>0$ there is a Lipschitz function $u_\eps \in \NX$ such that $\|u-u_\eps\|_\NX < \eps$.

Moreover, we can find measurable sets $E_\eps \subset \Pcal$ such that $u=u_\eps$ in $\Pcal \setminus E_\eps$ and $\meas{\bigcap_{\eps>0} E_\eps} = 0$. If both $\suplev{h}{\sigma}$ and $\suplev{u}{\sigma}$ are of finite measure for some $\sigma>0$, then we can require $\meas{E_\eps}<\eps$. (In particular, it suffices to assume that $u, h\in L^q_\fm$ or that $X \subset L^q_\fm\fcrim$ for some $q>0$.)
\end{thm}
%
%
Note that we will not use the doubling condition of $\mu$ in the proof, and indeed Theorem~\ref{thm:Haj_Lip-dens} holds even if the measure violates this condition. On the other hand, $\mu$ needs to be assumed non-atomic.
%
%
\begin{proof}
Since $\sigma \chi_{\suplev{h}{\sigma}} \le h \chi_{\suplev{h}{\sigma}}$ for every $\sigma \ge 0$, the absolute continuity of the quasi-norm of $X$ yields that $\|\sigma \chi_{\suplev{h}{\sigma}}\|_X \le \|h \chi_{\suplev{h}{\sigma}}\|_X \to 0$ as $\sigma \to \infty$ if $h \in X$.

Let $\eps>0$ and set $\eta = \eps / 6\cconc^2$, where $\cconc \ge 1$ is the modulus of concavity of the quasi-norm of $X$. Using Corollary~\ref{cor:bdd-dense-in-N1X}, we find $\sigma_0>1/\eps$ such that $\|u-v\|_\NX < \eta$, where $v$ is the truncation of $u$ at the levels $\pm\sigma_0$. Evidently, $u=v$ in $\Pcal \setminus \suplev{u}{\sigma_0}$. Moreover, if $\meas{\suplev{u}{\sigma}}<\infty$ for some $\sigma>0$,  then $0 = \meas{\bigcap_{\sigma>0} \suplev{u}{\sigma}} = \lim_{\sigma \to \infty} \meas{\suplev{u}{\sigma}}$. Therefore, we can choose $\sigma_0>0$ sufficiently large to obtain $\meas{\suplev{u}{\sigma_0}}<\eta$ in this case. Note that $h$ is a Haj\l{}asz gradient of $v$ as well.

Now, we will show that the weak estimate $\|\sigma \chi_{\suplev{h}{\sigma}}\|_X \to 0$ as $\sigma \to \infty$ yields that $h<\infty$ a.e.
Let $Q=\{x\in \Pcal: h(x) = \infty\}$. Then, $\|\sigma \chi_Q\|_X \le \|\sigma \chi_{\suplev{h}{\sigma}}\|_X \to 0$ as $\sigma \to \infty$. Thus, $\|\chi_Q\|_X = 0$, whence $\meas{Q} = 0$.

For an upper gradient $g\in X$ of $u$, we can find $\sigma_1\ge\sigma_0$ such that $\|g \chi_{\suplev{h}{\sigma_1}} \|_X < \eta$ by the absolute continuity of the quasi-norm of $X$. If $\meas{\suplev{h}{\sigma}} < \infty$ for some $\sigma>0$, then $0 = \meas{\bigcap_{\sigma>0} \suplev{h}{\sigma}} = \lim_{\sigma \to \infty} \meas{\suplev{h}{\sigma}}$. Therefore, we can choose $\sigma_1$ sufficiently large so that $\meas{\suplev{h}{\sigma_1}}<\eta$ in this case.

Now, fix $\sigma \ge \sigma_1$ such that $\| \sigma \chi_{\suplev{h}{\sigma}} \|_X < \eta$. 
Let $E\subset \Pcal$ be the exceptional set, where \eqref{eq:def-haj} fails, and let $A_\eta = E \cup \suplev{h}{\sigma}$. Thus, we obtain that $v|_{\Pcal \setminus A_\eta}$ is $2\sigma$-Lipschitz continuous, since $|v(x) - v(y)| \le \dd(x,y) (h(x) + h(y)) \le 2\sigma \dd(x,y)$. We define $u_\eps$ as the truncation of the upper McShane extension of $v|_{\Pcal \setminus A_\eta}$ at levels $\pm \sigma$, i.e.,
\[
  u_\eps (x) = \max\{-\sigma, \min\{ \sigma, \inf\{v(y) + 2\sigma \dd(x,y): y\in \Pcal \setminus A_\eta\} \} \} \quad \mbox{for }x\in\Pcal.
\]
As $\sup_{x\in\Pcal} |v(x)| \le \sigma_0 \le \sigma_1 \le \sigma$, we have 
\[
  \|v-u_\eps\|_X \le \| (|v| + \sigma) \chi_{A_\eta}\|_X \le 2\|\sigma \chi_{\suplev{h}{\sigma}}\|_X < 2\eta.
\]
Since $g\in X$ is an upper gradient of $u$, it is an upper gradient of $v$ as well. Then, $g+ 2\sigma$ is an upper gradient of $v-u_\eps$. Furthermore, it follows by Lemma~\ref{lem:glueing_lemma} that $(g+ 2\sigma) \chi_{A_\eta}$ is an $X$-weak upper gradient of $v-u_\eps$, whose minimal $X$-weak upper gradient can be estimated by
\[
  \| g_{v-u_\eps} \|_X \le \| (g+ 2\sigma) \chi_{A_\eta}\|_X \le \cconc( \|g \chi_{\suplev{h}{\sigma}}\|_X + 2  \|\sigma \chi_{\suplev{h}{\sigma}}\|_X) < 3\cconc \eta.
\]
Therefore,
\begin{align*}
  \|u-u_\eps\|_\NX & \le \cconc (\|u- v\|_\NX + \|v- u_\eps\|_\NX) \\
	& = \cconc (\|u- v\|_\NX + \|v- u_\eps\|_X + \| g_{v-u_\eps} \|_X)
    < \cconc (\eta + 2\eta + 3\cconc\eta) \le \eps.
\end{align*}
We see that $u_\eps = v$ outside of $A_\eta$ and $v = u$ outside of $\suplev{u}{\sigma_0}$, whence $u_\eps = u$ in $\Pcal \setminus E_\eps$, where $E_\eps = A_\eta \cup \suplev{u}{\sigma_0} = E \cup \suplev{h}{\sigma} \cup \suplev{u}{\sigma_0}$.
Both $\sigma$ and $\sigma_0$ depend on $\eps$ and $\sigma \ge \sigma_0 \to \infty$ as $\eps \to 0$. Thus, $\bigcap_{\eps>0} E_\eps = E\cup \bigcap_{\tau>0} (\suplev{h}{\tau} \cup \suplev{u}{\tau})$, which yields that $\meas{\bigcap_{\eps>0} E_\eps}=0$ since both $h$ and $u$ are finite a.e.
If the superlevel sets are of finite measure, then $\meas{E_\eps} \le \meas{\suplev{h}{\sigma}} + \meas{\suplev{u}{\sigma_0}} < 2 \eta < \eps$.

If $u,h \in L^p_\fm\fcrim$ for some $p>0$, then $\meas{\suplev{u}{\sigma} \cup \suplev{h}{\sigma}} \to 0$ as $\sigma \to \infty$ by Lemma~\ref{lem:superlevelsets-to-zero}.

Suppose now that $X\subset L^p_\fm$ for some $p>0$. Since $u\in X$, we have $u\in L^p_\fm$ and hence $\meas{\suplev{u}{\sigma}} \to 0$ as $\sigma \to \infty$ by Lemma~\ref{lem:superlevelsets-to-zero}. It remains to prove that $\suplev{h}{\sigma}$ is of finite measure for some $\sigma > 0$. Suppose on the contrary that $\meas{\suplev{h}{\sigma}} = \infty$ for all $\sigma>0$. Since $\Bigl\|\sigma^{1/p} \chi_{\suplev{h}{\sigma^{1/p}}}\Bigr\|_X \to 0$ as $\sigma\to \infty$, there is a sequence $\{\sigma_n\}_{n=1}^\infty\subset\Rbb^+$ such that $\sigma_n \ge n$ and $\Bigl\|\sigma_n^{1/p} \chi_{\suplev{h}{\sigma_n^{1/p}}}\Bigr\|_X \le (2\cconc)^{-n}$ for every $n\in\Nbb$. We choose a sequence of pairwise disjoint sets $F_n$ such that $\smash{F_n \subset \suplev{h}{\sigma_n^{1/p}}}$ and $\meas{F_n} = 1/n^2$. Let $f = \sum_{n=1}^\infty \sigma_n^{1/p} \chi_{F_n}$ and $F = \bigcup_{n=1}^\infty F_n$. Then,
\[
  \|f\|_X \le \sum_{n=1}^\infty \cconc^n \sigma_n^{1/p} \| \chi_{F_n} \|_X \le \sum_{n=1}^\infty \cconc^n \sigma_n^{1/p} \Bigl\| \chi_{\suplev{h}{\sigma_n^{1/p}}} \Bigr\|_X \le \sum_{n=1}^\infty \frac{1}{2^n} = 1.
\]
Hence, $f \in X$ but $f \notin L^p_\fm$ since $\meas{F} < \infty$ and
\[
  \| f \chi_F \|_{L^p}^p = \sum_{n=1}^\infty \sigma_n \meas{F_n} = \sum_{n=1}^\infty \frac{\sigma_n}{n^2} \ge \sum_{n=1}^\infty \frac{1}{n}= \infty,
\]
which contradicts the inclusion $X\subset L^p_\fm$. We have thus shown that $\meas{\suplev{h}{\sigma}} < \infty$ for some $\sigma>0$. Consequently, $\lim_{\sigma \to \infty} \meas{\suplev{h}{\sigma}} = \meas{\bigcap_{\sigma>0} \suplev{h}{\sigma}} = 0$.
\end{proof}
%
%
Note that the hypotheses in Theorem~\ref{thm:Haj_Lip-dens} are sufficient but not necessary by any means. We saw in Proposition~\ref{pro:NX=X-dens} that the density of Lipschitz functions in $\NX$ relies only on the properties of $X$ if the Newtonian space is trivial (i.e., if $\NX = X$).

Another tool to study Lipschitz and H\"older continuity of Sobolev (thus Newtonian) functions was introduced by Calder\'on and Scott~\cite{CalSco} in 1978, cf.~Calder\'on~\cite{Cal}.
%
%
\begin{df}
Let $u \in L^1_\loc(\Pcal)$. Then, for $\alpha \in (0,1]$, we define the \emph{fractional sharp maximal function} by
\[
  u^\sharp_\alpha(x) = \sup_{r>0} \frac{1}{r^\alpha} \fint_{B(x,r)} |u - u_{B(x,r)}|\,d\mu, \quad x\in\Pcal.
\]
\end{df}
Roughly speaking, $u^\sharp_\alpha$ measures the $\alpha$-H\"older continuity of a function. Since we are interested in Lipschitz continuity, we will only work with $u^\sharp_1$.
\begin{rem}
If $u$ is $L$-Lipschitz continuous, then obviously $u^\sharp_1 \le 2L$. The converse also holds true. Namely, if a function $u\in L^1_\loc$ has $u^\sharp_1 \in L^\infty$, then there is a Lipschitz continuous function $\tilde{u}$ such that $u=\tilde{u}$ a.e. Boundedness of $u^\sharp_1$ guarantees that $u$ has a Haj\l asz gradient $h\in L^\infty$, which was shown by Haj\l asz and Kinnunen~\cite[Lemma~3.6]{HajKin}. Let $L = \|h\|_{L^\infty}$ and $E_L = \suplev{h}{L} \cup E$, where $E$ is the set where \eqref{eq:def-haj} fails. Then, $u|_{\Pcal\setminus E_L}$ is $2L$-Lipschitz and it has a unique continuous extension to $\Pcal$ since $\Pcal\setminus E_L$ is dense in $\Pcal$. Such an extension retains the $2L$-Lipschitz continuity.
\end{rem}
\begin{thm}
\label{thm:general_sharp_Lip-dens}
Assume that $X$ is a quasi-Banach function lattice with absolutely continuous quasi-norm. Let $u \in \NX$ and suppose that the fractional sharp maximal function $v^\sharp_1$ satisfies the weak estimate $\|\sigma \chi_{\suplevshp{v}{\sigma}} \|_X \to 0$ as $\sigma \to \infty$ for every truncation $v$ of $u$, where
\[
  \suplevshp{v}{\sigma} \coloneq \suplev{v^\sharp_1}{\sigma} = \{ x\in \Pcal: {v}^\sharp_1(x) > \sigma \} \quad \mbox{for }\sigma \ge 0.
\]
(In particular, it suffices that $\|\sigma \chi_{\suplevshp{u}{\sigma}} \|_X \to 0$ as $\sigma \to \infty$.) 
Then, for every $\eps>0$ there is a Lipschitz function $u_\eps \in \NX$ such that $\|u-u_\eps\|_\NX < \eps$.
\end{thm}
Similarly as before, we can find $u_\eps=u$ outside of a set of arbitrarily small measure provided that there are some superlevel sets of $u$ and $u^\sharp_1$ (or $v^\sharp_1$) of finite measure.
%
%
\begin{proof}
Whenever $v$ is a truncation of $u$, we have $|v(x) - v(y)| \le |u(x) - u(y)|$ for every $x,y \in \Pcal$, whence
\begin{multline*}
  \fint_B |v(x)-v_B|\,d\mu(x)  \le \fint_B \fint_B |v(x) - v(y)| \,d\mu(x)\,d\mu(y) \\
	 \le \fint_B \fint_B |u(x) - u(y) + u_B - u_B| \,d\mu(x)\,d\mu(y) \le 2 \fint_B |u(x) - u_B|\,d\mu(x).
\end{multline*}
Therefore, $v^\sharp_1 \le 2 u^\sharp_1$ and if $u^\sharp_1$ satisfies the weak estimate, then so does $v^\sharp_1$.

Let $\eps > 0$. Then, there is a truncation $v\in\NX$ of $u$ such that $\|u-v\|_\NX < \eps/2\cconc$ by Corollary~\ref{cor:bdd-dense-in-N1X}.
Applying~\cite[Lemma~3.6]{HajKin}, we see that $c v^\sharp_1$ is a Haj\l asz gradient of~$v$ for some $c = c(c_\dbl) > 0$. Thus, there is a Lipschitz function $u_\eps \in \NX$ such that $\|u_\eps - v\|_\NX < \eps/2\cconc$ by Theorem~\ref{thm:Haj_Lip-dens}. Finally, the triangle inequality yields
\[
  \| u-u_\eps\|_\NX \le \cconc (\|u-v\|_\NX + \|v-u_\eps\|_\NX) < \eps.
\qedhere
\]
\end{proof}
%
%
In the previous proof, we have used that a multiple of the fractional sharp maximal function $u^\sharp_1$ is a Haj\l asz gradient of a function $u\in L^1_\loc$. On the other hand, if $u \in L^1_\loc$ has a Haj\l asz gradient $h$, then it is easy to show that $u^\sharp_1 \le 4 M_1^c h$, where $M_1^c$ is the centered Hardy--Littlewood maximal operator. Note that this estimate holds true even if $\mu$ is not doubling.

Similarly as with the Haj\l{}asz gradients, it is in general impossible without any additional assumptions on the metric measure space to find (or at least provide an estimate for) the fractional sharp maximal function $u^\sharp_1$ using an ($X$-weak) upper gradient of $u\in\NX$. A clear connection, perhaps not optimal, is however obtained if $\Pcal$ supports a $p$-Poincar\'e inequality (see Definition~\ref{df:pPI} above).
\begin{thm}
\label{thm:general_Lip-dens}
Assume that $\Pcal$ is a $p$-Poincar\'e space for some $p \in [1, \infty)$. Suppose further that $X$ is a quasi-Banach function lattice with absolutely continuous quasi-norm and that $\|\sigma \chi_{\suplevp{p}{v}{\sigma}} \|_X \to 0$ as $\sigma \to \infty$ whenever $v\in X$, where $\suplevp{p}{v}{\sigma}$ is the superlevel set of $M_pv$. Then, the set of Lipschitz functions is dense in $\NX$.
\end{thm}
%
%
Similarly as before, the approximating Lipschitz functions coincide with the approximated Newtonian functions outside of sets of arbitrarily small measure provided that $X \subset L^q_\fm$ for some $q>0$.
%
%
\begin{proof}
Since $\Pcal$ is a $p$-Poincar\'e space, we obtain that $u^\sharp_1(x) \le 2 c_{\PI} M_p g(x)$ whenever $g\in X$ is an upper gradient of $u\in \NX$. Since $\bigl\|\sigma \smash{\chi_{\suplevp{p}{g}{\sigma}}} \bigr\|_X \to 0$ as $\sigma \to \infty$, the fractional sharp maximal function $u^\sharp_1$ satisfies the weak estimate of Theorem~\ref{thm:general_sharp_Lip-dens}, which then yields the desired conclusion.
\end{proof}
The following example shows that the hypotheses that $\Pcal$ supports a $p$-Poincar\'{e} inequality and that $M_p$ obeys the weak estimate are in fact more restrictive than posing an analogous assumption that $u^\sharp_1$ satisfies the weak estimate of Theorem~\ref{thm:general_sharp_Lip-dens} for every $u\in \NX$. 
\begin{exa}
Consider the bow-tie in ($\Rbb^n$, $dx$), i.e., let
\[
  \Pcal = \bigl\{(x_1, x_2, \ldots, x_n) \in \Rbb^n: x_i x_j\ge 0 \mbox{ for all } i,j = 1,\ldots,n\bigr\}.
\]
Let $X=L^q(\Pcal)$ for some $q\in[1, \infty)$. In fact, we are revisiting Bj\"{o}rn, Bj\"{o}rn, and Shanmugalingam~\cite[Example 5.2]{BjoBjoSha}, where other methods were used to show that Lipschitz functions are dense in $\NX$ even though $\Pcal$ is a $p$-Poincar\'{e} space if and only if $p>n$ (see also~\cite[Example A.23]{BjoBjo}).

Theorem~\ref{thm:general_Lip-dens} yields merely that Lipschitz functions are dense in $N^{1,q}\coloneq N^1L^q$ for $q>n$. We will show that the hypotheses of Theorem~\ref{thm:general_sharp_Lip-dens} are fulfilled for every $u\in N^{1,q}$ with $q\in[1, n)$ as well, yielding density of Lipschitz functions in $N^{1,q}$ for every $q\in[1, \infty)\setminus\{ n \}$.

Let $q\in [1, n)$. We can split $\Pcal = \Pcal^+ \cup \Pcal^-$, where
\[
  \Pcal^+ = \bigl\{ x\in\Rbb^n: x_j \ge 0,\  j=1,\ldots,n\bigr\} \quad \mbox{and} \quad \Pcal^- = \bigl\{ x\in\Rbb^n: x_j \le 0,\ j=1,\ldots,n \bigr\}.
\]
Both $\Pcal^+$ and $\Pcal^-$ support a $1$-Poincar\'e inequality, e.g., by~\cite[Example 5.6]{BjoBjo}. Let $v$ be a truncation of $u \in N^{1,q}$ and let $g\in L^q$ be an upper gradient of $u$ and thus of $v$. Let $x \in \Pcal$ and $B = B(x,r)$. Then,
\[
  \fint_{B} |v - v_B| \,d\mu \lesssim
    r \fint_{B} g\,d\mu \le r M_1 g(x)\quad\mbox{if }r\le |x|.
\]
Suppose now that $r > |x|$ and $x\in \Pcal^+$. By the triangle inequality, we obtain that
\begin{multline*}
  \fint_{B} |v - v_B| \,d\mu  \lesssim \fint_B |v-v_{B\cap \Pcal^+}|\,d\mu \\
   \le \fint_{B\cap \Pcal^+} |v-v_{B\cap \Pcal^+}|\,d\mu + \fint_{B\cap \Pcal^-}|v-v_{B\cap \Pcal^+}|\,d\mu 
  \lesssim r \fint_{B\cap \Pcal^+} g\,d\mu + \|v\|_{L^\infty} \,.
\end{multline*}
Hence, $v^\sharp_1(x) \lesssim M_1 g(x) + \|v\|_{L^\infty} / |x|$ whenever $x\in \Pcal^+$. An analogous argument shows that the inequality holds for $x\in\Pcal^-$ as well. Therefore, there is $c>0$ such that
\begin{multline*}
 \suplevshp{v}{c \sigma} \subset \biggl\{ x \in \Pcal: M_1 g(x) + \frac{\|v\|_{L^\infty}}{|x|} > \sigma \biggr\} \\
 \subset
   \biggl\{ x \in \Pcal: M_1 g(x) > \frac{\sigma}{2} \biggr\}  \cup
	 \biggl\{ x \in \Pcal: \frac{\|v\|_{L^\infty}}{|x|} > \frac{\sigma}{2} \biggr\} = \Suplevp{1}{g}{\frac{\sigma}{2\vphantom{\|_L}}} \cup \Suplev{h}{\frac{\sigma}{2\|v\|_{L^\infty}}}\,,
\end{multline*}
where $h(x) = 1/|x|$ for $x\in \Pcal$.
The function $M_1 g$ fulfills the needed weak estimate by~\cite[Lemma~3.12 and Theorem~3.13]{BjoBjo} (see also Section~\ref{sec:weaktype} below). The superlevel sets  $\suplev{h}{\tilde\sigma}$ are balls of radius $1/\tilde\sigma$, centered at the origin.
Therefore, $\| \tilde\sigma \chi_{\suplev{h}{\tilde\sigma}} \|_{L^q} \approx \tilde\sigma^{1-n/q} \to 0$ as $\tilde\sigma \to \infty$. Consequently, $\| \sigma \chi_{\suplevshp{v}{\sigma}}\|_{L^q} \to 0$ as $\sigma \to \infty$. Note that the rate of convergence depends on $\|v\|_{L^\infty}$, i.e., on the chosen truncation of $u$. Theorem~\ref{thm:general_sharp_Lip-dens} now gives that $u$ can be approximated in $N^{1,q}$ by Lipschitz functions.

The case $q=n$ is more delicate. In general, we obtain merely that $\bigl\| \sigma \chi_{\suplevshp{v}{\sigma}}\bigr\|_X$ is bounded but does not tend to zero as $\sigma \to \infty$. For example, such a behavior is exhibited by $v(x)\coloneq(\chi_{\Pcal^+}(x)-\dist(B(0,1),x))^+ \in N^{1,n}$. Nevertheless, Lipschitz functions are dense even in $N^{1,n}$, which was shown in~\cite[Example~5.2]{BjoBjoSha}.
\end{exa}
%
%
The following proposition extends known density results in the variable exponent Sobolev and Newtonian spaces on $\Rbb^n$, cf.\@ Diening, Harjulehto, H\"ast\"o and R\r{u}\v{z}i\v{c}ka~\cite[Theorem~9.5.2]{DieHarHasRuz} and Harjulehto, H\"ast\"o and Pere~\cite[Theorem 3.5]{HarHasPer}, respectively. The main difference, when using our approach via the weak type estimate for the maximal operator, is that we allow for $p^- = \essinf_{x\in\Rbb^n} p(x) = 1$.
\begin{pro}
Suppose that $(\Pcal, \mu) = (\Rbb^n, dx)$. Let $X$ be the variable exponent Lebesgue space $L^{p(\cdot)}$ whose norm is given by
\[
  \| u\|_{p(\cdot)} = \inf\biggl\{ \lambda > 0 : \int_{\Rbb^n} \biggl( \frac{|u(x)|}{\lambda}\biggr)^{p(x)}\,dx \le 1\biggr\},
\]
where $p: \Rbb^n \to [1, \infty)$ is measurable. Assume that $p$ is essentially bounded and that $p$ is of class $\Acal$, i.e.,
\[
  \biggl\|\sum_{Q \in \Qcal} \Bigl(\chi_Q \fint_Q |f(x)|\,dx\Bigr) \biggr\|_{p(\cdot)} \lesssim \|f\|_{p(\cdot)}
\]
holds uniformly for all $f\in L^{p(\cdot)}$ and all systems of pairwise disjoint cubes $\Qcal$, cf.\@~\cite[Definition 4.4.6]{DieHarHasRuz}. 
Then, the \emph{Lipschitz truncations}, i.e., bounded Lipschitz functions that coincide with a given function outside of sets of small measure, are dense in $N^{1,p(\cdot)}(\Rbb^n)$.
\end{pro}
Theorem 4.4.8 of~\cite{DieHarHasRuz} yields that $p$ is in particular of class $\Acal$ if $p$ is globally log-H\"older continuous, i.e., if $|p(x) - p(y)| \lesssim -1/\log(|x-y|)$ whenever $|x-y|<1/2$ and if there is $p_\infty \in [1,\infty)$ such that $|p(x) - p_\infty| \lesssim 1/\log (e+|x|)$ for all $x\in\Rbb^n$.
\begin{proof}
The space $\Rbb^n$ with the Lebesgue $n$-dimensional measure supports a $1$-Poincar\'{e} inequality. By~\cite[Theorem 3.4.1]{DieHarHasRuz}, the $L^{p(\cdot)}$ norm is absolutely continuous if and only if $p$ is essentially bounded. It is also shown in~\cite[Theorem 4.4.10]{DieHarHasRuz}, that if $p$ is of class $\Acal$, then the maximal operator $M_1$ is of weak type $(p(\cdot),p(\cdot))$, i.e.,
\[
  \sup_{\sigma>0} \Bigl\| \sigma \chi_{\suplevp{1}{f}{\sigma}}\Bigr\|_{p(\cdot)} \lesssim \|f\|_{p(\cdot)},
\]
where $\suplevp{1}{f}{\sigma} = \{x\in\Rbb^n: M_1f(x) > \sigma\}$ as before. It remains to show that in fact $\Bigl\| \sigma \chi_{\suplevp{1}{f}{\sigma}}\Bigr\|_{p(\cdot)} \to 0$ as $\sigma \to \infty$. Let $f\in L^{p(\cdot)}$ be fixed and then define $f_\sigma = f \chi_{\suplev{f}{\sigma/2}}$ for $\sigma > 0$. Then, $M_1f \le M_1f_\sigma + \sigma/2$ whence $\suplevp{1}{f}{\sigma} \subset \suplevp{1}{f_\sigma}{\sigma/2}$. Consequently,
\[
  \Bigl\| \sigma \chi_{\suplevp{1}{f}{\sigma}}\Bigr\|_{p(\cdot)} \le \Bigl\| \sigma \chi_{\suplevp{1}{f_\sigma}{\sigma/2}}\Bigr\|_{p(\cdot)} \lesssim \| f_\sigma \|_{p(\cdot)} = \bigl\| f \chi_{\suplev{f}{\sigma/2}} \bigr\|_{p(\cdot)} \to 0\quad\mbox{as }\sigma \to \infty.
\]
In this estimate, we have used that $\measl{\bigcap_{\sigma>0} \suplev{f}{\sigma/2}}=0$ so that the absolute continuity of the norm yields zero as the limit.
Theorem~\ref{thm:general_Lip-dens} and its proof give the desired conclusion of density of Lipschitz functions in $N^{1,p(\cdot)}(\Rbb^n)$.
\end{proof}
%
%
%
%
\section{Rearrangement-invariant spaces}
\label{sec:rispaces}
In order to be able to study boundedness of the maximal operators $M_p$, some structure of the function space $X$ needs to be known. In the current paper, we discuss a rather wide class of function spaces where the function norm is, roughly speaking, invariant under measure-preserving transformations. The setting of these so-called \ri spaces includes among others the Lebesgue $L^p$ spaces, the Orlicz $L^\Psi$ spaces, and the Lorentz $L^{p,q}$ spaces.

A quasi-normed function lattice $X = X(\Pcal, \mu)$ is \emph{rearrangement-invariant} if it satisfies the condition
\begin{enumerate}
  \renewcommand{\theenumi}{(RI)}
  \item \label{df:RI}
  if $u$ and $v$ are \emph{equimeasurable}, i.e.,
  \[
    \meas{\{x\in \Pcal: u(x) > t\}} = \meas{\{x\in \Pcal: v(x) > t\}} \quad\mbox{for all } t\ge0,
  \]
  then $\|u\|_X = \|v\|_X$.
\end{enumerate}
We say that $X$ is an \emph{\ri space} if it is a rearrangement-invariant Banach function space. In other words, $X$ satisfies not only~\ref{df:qBFL.initial}--\ref{df:BFL.locL1} with the modulus of concavity $\cconc = 1$, but also~\ref{df:RI}.

It is easy to show that if $X \emb Y_\fm$, where both $X$ and $Y$ are \ri spaces over $(\Pcal, \mu)$, then the constant $c_E\ge 0$ in the embedding inequality $\| u \chi_E\|_Y \le c_E \|u\|_X$ actually depends only on $\meas{E}$ for all measurable sets $E \subset \Pcal$ of finite measure.

For $f \in \Mcal(\Pcal, \mu)$, we define its \emph{distribution function $\mu_f$} by
\[
  \mu_f(t) = \meas{\{x\in \Pcal: |f(x)| > t\}}, \quad t\in[0, \infty).
\]
%
%
Furthermore, we define the \emph{decreasing rearrangement $f^*$} of $f$ as the right-continuous generalized inverse function of $\mu_f$, i.e.,
\[
  f^*(t) = \inf\{ s\ge 0: \mu_f(s) \le t\}, \quad t\in[0, \infty).
\]
%
%
The Cavalieri principle implies that $\|f\|_{L^1(\Pcal, \mu)} = \|\mu_f \|_{L^1(\Rbb^+, \lambda^1)} = \|f^*\|_{L^1(\Rbb^+, \lambda^1)}$. The \emph{elementary maximal function} $f^{**}$ of $f$ is given by
\[
  f^{**}(t) = \fint_0^t f^*(s)\,ds, \quad t\in \Rbb^+.
\]

For a measure space $(\Rcal, \nu)$, a function $u \in \Mcal(\Rcal, \nu)$, and $t \in [0, \nu(\Rcal))$, we can find a measurable ``superlevel'' set $A \subset \Rcal$ such that $\nu(A) = t$ and $|u(x)| \ge |u(y)|$ whenever $x\in A$ and $y\in \Rcal \setminus A$.
In general, such a set $A$ is not defined uniquely by these conditions. Hence, we define $\suplevr{u}{t}$ as the family of all measurable sets $A$ with $\nu(A) = t$ that obey
\begin{equation}
  \label{eq:df-suplevr}
	\{x \in \Rcal: |u(x)| > u^*(t)\} \subset A \subset \{x \in \Rcal: |u(x)| \ge u^*(t)\}.
\end{equation}
Depending on the context, we will use either $(\Rcal, \nu) = (\Rbb^+, \lambda^1)$ or $(\Rcal, \nu) = (\Pcal, \mu)$.
%
%
\begin{df}
Given a quasi-normed rearrangement-invariant function lattice $X$, we define the \emph{fundamental function} of $X$ as 
\[
  \phi_X(t) =
    \begin{cases}
      \|\chi_{E_t}\|_X, & t\in [0, \meas{\Pcal}),\\
      \|1\|_X,          & t\in[\meas{\Pcal}, \infty),
    \end{cases}
\]
where $E_t \subset \Pcal$ is an arbitrary measurable set with $\meas{E_t} = t$.
\end{df}
%
%
The purpose of defining $\phi_X$ beyond $\meas{\Pcal}$ is merely for the sake of convenience, which will allow us to skip the distinction of the exact (possibly infinite) value of $\meas{\Pcal}$ in the coming claims and proofs. We will simply write $\phi$ instead of $\phi_X$ whenever any confusion of function spaces is unlikely to arise. Different spaces may very well have the same fundamental function, which is seen in the example below.
%
%
\begin{exa}
(a) For the Lebesgue $L^p$ spaces, $1\le p<\infty$, we obtain that $\phi(t) = t^{1/p}$ for $t<\meas{\Pcal}$. It is also easy to see that $\phi_{L^\infty} = \chi_{(0, \infty)}$.

(b) The \emph{Lorentz spaces} $L^{p,q}(\Pcal)$ and $L^{p,\infty}(\Pcal)$ for $1\le p, q < \infty$, whose respective (quasi)norms are defined by
\[
   \| u \|_{L^{p,q}} = \biggl( \frac{q}{p} \int_0^\infty (u^*(t) t^{1/p})^q  \frac{dt}{t} \biggr)^{1/q}\quad\mbox{and}\quad \| u \|_{L^{p,\infty}} = \sup_{t > 0} u^*(t) t^{1/p},
\]
have the fundamental function $\phi(t) = t^{1/p}$ for $t<\meas{\Pcal}$.

(c) The fundamental functions of the grand and small Lebesgue spaces, which arise in the extrapolation theory, are for $t$ near zero estimated by
\[
  \phi_{L^{p)}}(t) \approx \frac{t^{1/p}}{|{\log t}|} \quad\mbox{and}\quad \phi_{L^{(p}}(t) \approx t^{1/p} |{\log t}|,
\]
which was established by Lang and Pick, see Capone and Fiorenza~\cite{CapFio}.

(d) The Orlicz spaces based on an $N$-function $\Psi$ with the Luxemburg norm
\[
  \| u \|_{L^\Psi} = \inf \biggl\{\lambda>0: \int_\Pcal \Psi\biggl(\frac{|u(x)|}{\lambda}\biggr)\,d\mu(x) \le 1\biggr\}
\]
have the fundamental function $\phi(t) = 1/\Psi^{-1}(1/t)$ for $0<t< \meas{\Pcal}$.
\end{exa}
%
%
A function $f: [0, \infty) \to [0, \infty)$ is called \emph{quasi-concave on $[0, R)$} for some $R>0$, if it satisfies:
\begin{itemize}
	\item $f(0) = 0 < f(t)$ for $t\in (0, R)$,
	\item $f(t)$ is increasing for $t\in[0, R)$,
	\item $f(t) / t$ is decreasing for $t\in(0, R)$.
\end{itemize}
If $R=\infty$, we say simply that $f$ is \emph{quasi-concave}.

Note that a function $f$ that is quasi-concave on $[0,R)$ for some $R>0$ is Lipschitz (and hence absolutely continuous) on $[\delta, R)$ for every $\delta > 0$. The Lipschitz constant is at most $f(\delta)/\delta$ then. Furthermore, there exists a concave function $\tilde{f}$ such that $\tilde{f}/2 \le f \le \tilde{f}$ on $[0,R)$, cf.\@~\cite[Proposition II.5.10]{BenSha}.

If a function $f: [0, \infty) \to [0, \infty)$ is increasing and concave on $[0, R)$ and if $f(0) = 0 < f(t)$ for all $t\in (0,R)$, then $f$ is quasi-concave on $[0, R)$ since
\begin{equation}
  \label{eq:concave-quasiconcave}
  \frac{f(t)}{t} = \frac{f(t) - f(0)}{t-0} \le \frac{f(s) - f(0)}{s-0} = \frac{f(s)}{s}\quad\mbox{for } 0<s<t<R.
\end{equation}
It is shown in~\cite[Corollary II.5.3]{BenSha} that the fundamental function $\phi$ of an \ri space $X$ is quasi-concave. By~\cite[Proposition II.5.11]{BenSha}, every \ri space can be equivalently renormed so that the fundamental function is concave. If $X$ has an absolutely continuous norm, then $\phi(0\limplus) \coloneq \lim_{t\to 0\limplus} \phi(t) = 0$. Note however that the converse does not hold true in general. For example, the weak-$L^p$ spaces (i.e., $L^{p,\infty}$) satisfy $\phi(0\limplus) = 0$ if $p<\infty$ even though their quasi-norm lacks the~\ref{df:AC} property.
%
%
\begin{df}
For a quasi-concave function $\phi$, we define the \emph{(classical) Lorentz space} $\Lambda_\phi^q$, where $q\in [1, \infty)$, the \emph{Marcinkiewicz space} $M_\phi$ and the \emph{weak Marcinkiewicz space} $M^*_\phi$ by their respective (quasi)norms:
\begin{align*}
  \|u\|_{\Lambda^q_\phi} & = \biggl( \int_0^\infty (u^*(t) \phi(t))^q \frac{dt}{t} \biggr)^{1/q}, \\
  \|u\|_{M_\phi} &= \sup_{t>0} u^{**}(t) \phi(t),\\
  \|u\|_{M^*_\phi} &= \sup_{t>0} u^{*}(t) \phi(t).
\end{align*}
If $\phi$ is an increasing concave function, we define the \emph{Lorentz space} $\Lambda_\phi$ via its norm
\[
  \|u\|_{\Lambda_\phi} = \int_{[0, \infty)} u^*(t)\,d\phi(t) = \phi(0\limplus) \|u\|_{L^\infty} +  \int_0^\infty u^*(t) \phi'(t) \,dt,
\]
where $d\phi(t) = \phi'(t)\,dt$ a.e.\@ on $\Rbb^+$ due to the absolute continuity of $\phi$.

Given an \ri space $X$ with fundamental function $\phi$ (as long as $X$ is considered renormed so that $\phi$ is concave if needed), we write $\Lambda(X)$, $\Lambda^q(X)$, $M(X)$, and $M^*(X)$ instead of $\Lambda_\phi$, $\Lambda^q_\phi$, $M_\phi$, and $M^*_\phi$, respectively.
\end{df}
%
%
Neither the notation, nor the naming of these spaces is unified in the literature. Both $\Lambda_\phi$ and $M_\phi$ are sometimes called Lorentz spaces. Both $M_\phi$ and $M^*_\phi$ may very well be called weak Lorentz or just Marcinkiewicz spaces. 
For instance, our $\Lambda_\phi$ is denoted by
\[
  \Lambda(w, 1), \quad \Lambda_1(w), \quad \Lambda^1(w), \quad \Lambda_\phi, \quad \mbox{and}\quad L(\widetilde{w}, 1)
\]
in Lorentz~\cite{Lor}, Sawyer~\cite{Saw}, Cwikel, Kami\'{n}ska, Maligranda, and Pick~\cite{CwiKamMalPic}, Bennett and Sharpley~\cite{BenSha}, and Sparr~\cite{Spa}, respectively, where $w(t) = \phi'(t)$ and $\widetilde{w}(t) = t \phi'(t)$ for $t>0$.
Furthermore, our $M_\phi$ is denoted by $\Lambda^*(\psi', 1)$, $\Gamma^{1, \infty}(w)$, and $M_\phi$ in~\cite{Lor},~\cite{CwiKamMalPic}, and~\cite{BenSha}, respectively, where $\psi$ is the associated fundamental function of $\phi$, i.e., $\psi(t) = t/\phi(t)$.
The space $\Lambda^q_\phi$ has been studied in~\cite{Lor},~\cite{CwiKamMalPic}, and~\cite{Spa} using the notations $\Lambda(w_q, q)$, $\Lambda^q(w_q)$, and $L(\phi, q)$, respectively, where $w_q(t) = \phi(t)^q/t$.
Finally, our $M^*_\phi$ appears in~\cite{Spa} and~\cite{CwiKamMalPic} as $L(\phi, \infty)$ and $\Lambda^{1, \infty}(w)$, respectively. Besides, the notation $M^*(X)$ can be found in~\cite{BenSha}. The interested reader may consult~\cite{CwiKamMalPic}, where various references on Lorentz and Lorentz-type spaces are provided.
%
%
\begin{exa}
Focusing on the Lebesgue spaces, we can see that
\begin{enumerate}
	\item $\Lambda(L^1) = M(L^1) = L^1$, $\Lambda^q(L^1) = L^{1,q}$, and $M^*(L^1) = L^{1,\infty}$;
	\item $\Lambda(L^p) = L^{p,1}$, $\Lambda^q(L^p) = L^{p,q}$, and $M(L^p) = M^*(L^p) = L^{p,\infty}$, whenever $p\in (1, \infty)$,
	\item $\Lambda(L^\infty) = M(L^\infty) = M^*(L^\infty) = L^\infty$, whereas $\Lambda^q(L^\infty) = \{0\}$.
\end{enumerate}
\end{exa}

If $X$ is an \ri space, then $\Lambda(X)$ and $M(X)$ are \ri spaces by~\cite[Proposition II.5.8, Theorem II.5.13]{BenSha}. In general, $M^*(X)$ is merely a rearrangement-invariant quasi-Banach function lattice.  They all have the same fundamental function as $X$ and
\begin{equation}
  \label{eq:LambdaX-X-MX-embed}
  \Lambda(X) \emb X \emb M(X) \emb M^*(X)
\end{equation}
with the embedding norms equal to $1$.

If $\phi^q$ is quasi-concave, then $\Lambda^q_\phi$ is an \ri space. The triangle inequality in this case follows from Lorentz~\cite[Theorem 1]{Lor}. Otherwise, the space may be merely quasi-normed (see Sparr~\cite[Theorem 1.2]{Spa}). The fundamental function of $\Lambda^q_\phi$ is different from $\phi$ unless $\phi(t)=t^{1/q}$ for $0\le t<\meas{\Pcal}$, in which case $\Lambda^q_\phi = L^q$. It is however comparable to $\phi$ provided that $\phi'(t) \approx \phi(t)/t$ for $0<t<\meas{\Pcal}$, which occurs, e.g., if $\phi^\alpha$ is a convex function or, more generally, if $\phi(t)^\alpha/t$ is increasing on $(0, \meas{\Pcal})$ for some $\alpha \ge 1$. 

The classical Lorentz spaces associated with the same quasi-concave function $\phi$ are embedded into each other relative to the exponent. The precise statement is given in the following lemma, which in fact follows from Stepanov~\cite[Proposition~1]{Ste}. In order to check the hypotheses of that proposition, a similar calculation as in the proof below is needed (with $v^*=\chi_{(0,a)}$ for arbitrary $a>0$). Therefore, we present a simple direct proof of the embedding, which is an elementary modification of the proofs available in the literature, where only $\phi(t) = t^\alpha$ with some $\alpha\in[0,\infty)$ is considered, cf.\@~\cite[Proposition IV.4.2]{BenSha}.
%
%
\begin{lem}
\label{lem:lorentz-embedding}
Let $\phi$ be a quasi-concave function and suppose that $1\le q < p < \infty$. Then, $\Lambda^q_\phi \emb \Lambda^p_\phi$ and the norm of the embedding can be estimated independently of $\phi$.
\end{lem}
%
%
\begin{proof}
Let $v \in \Lambda^q_\phi$. First, we show that $v \in M^*_{\phi}$ using the relation $\phi'(s) \le \phi(s)/s$ for $s>0$, which follows from the quasi-concavity of $\phi$,
\begin{align*}
   &\|v\|_{M^*_\phi} = \biggl(\sup_{t>0} v^*(t)^q \phi(t)^q\biggr)^{1/q} = \biggl(\sup_{t>0}  v^*(t)^q \int_0^t q \phi(s)^{q} \frac{\phi'(s)}{\phi(s)}\,ds\biggr)^{1/q} \\
    &\quad \le c_q \biggl(\sup_{t>0} v^*(t)^q \int_0^t \phi(s)^q \frac{ds}{s}\biggr)^{1/q} \le c_q \biggl(\sup_{t>0}   \int_0^t (v^*(s)\phi(s))^q \frac{ds}{s}\biggr)^{1/q} = c_q\|v\|_{\Lambda_\phi^q},
\end{align*}
where $c_q = q^{1/q}$. Now, we can estimate
\begin{align*}
 \|v\|_{\Lambda^p_\phi}^p & = \int_0^\infty (v^*(t) \phi(t))^p \frac{dt}{t} \le \sup_{t>0} (v^*(t) \phi(t))^{p-q} \int_0^\infty (v^*(t) \phi(t))^q \frac{dt}{t} \\
 & = \|v\|_{M^*_\phi}^{p-q} \|v\|_{\Lambda^q_\phi}^q \le c_q^{p-q} \|v\|_{\Lambda^q_\phi}^p.
\end{align*}
Thus, we obtain the desired inequality $\|v\|_{\Lambda^p_\phi} \le c_{p,q} \|v\|_{\Lambda^q_\phi}$, where $c_{p,q} = q^{1/q - 1/p}$.
\end{proof}
%
%
Given an \ri space $X$ over $(\Pcal, \mu)$, there is an \ri space $\reps{X}$ over $(\Rbb^+, \lambda^1)$, the so-called \emph{representation space} of $X$, such that $\|u\|_X = \|u^*\|_{\reps X}$ for all $u\in\Mcal(\Pcal, \mu)$. Existence of such a space $\reps X$ is established by the \emph{Luxemburg representation theorem} (see~\cite[Theorem II.4.10]{BenSha}). For the sake of uniqueness, $\reps X$ may be chosen such that $\|f\|_{\reps X} = \|f^* \chi_{(0, \meas{\Pcal})}\|_{\reps X}$ for all $f\in \reps X$. Furthermore, $\phi_X = \phi_{\reps X}$, whence $\Lambda(\reps X) = \reps{\Lambda(X)}$, $M(\reps X) = \reps{M(X)}$, and $M^*(\reps X) = \reps{M^*(X)}$.

The next lemma shows another rather unsurprising fact, namely, that the norm of the representation space retains the absolute continuity.
\begin{lem}
If an \ri space $X$ has an absolutely continuous norm, then so does $\reps X$.
\end{lem}
\begin{proof}
Let $f \in \reps X$. Then, there is a non-negative function $u\in X$ such that $u^* = f^*$ by~\cite[Corollary~II.7.8]{BenSha}.  Let $\{E_n\}_{n=1}^\infty$ be a decreasing sequence of sets in $(0, \meas{\Pcal})$ such that $\measl{\bigcap_{n=1}^\infty E_n} = 0$. Let $\eps > 0$ be arbitrary. Since $\bigcap_{R=1}^\infty (\Pcal \setminus B(z, R)) = \emptyset$ for any fixed $z\in\Pcal$, we can find a ball $B\subset\Pcal$ such that $\|u\chi_{\Pcal \setminus B}\|_X < \eps/2$. Choose $F \in \suplevr{f}{\meas{B}}$ arbitrarily (recall that $\suplevr{f}{t}$ was defined in \eqref{eq:df-suplevr} as the collection of all measurable ``superlevel'' sets of $f$ whose measure is equal to $t\ge0$).

Let $E_n' = E_n \cap F$. For every $n\ge 1$, choose $G_n \in\suplevr{u}{\measl{E'_n}}$ such that $G_n \supset G_{n+1}$. Hence, $\meas{G_n}\to 0$ as $n\to\infty$. Thus, there is $n_0 \ge 1$ such that $\|u \chi_{G_n}\|_X < \eps /2$ for every $n\ge n_0$. For such $n$, we can estimate
\begin{align*}
  \| f \chi_{E_n}\|_{\reps X} & 
  \le \| f \chi_{E'_n}\|_{\reps X} + \| f \chi_{E_n\setminus F}\|_{\reps X} 
  \le \| f \chi_{E'_n}\|_{\reps X} + \| f \chi_{(0, \meas{\Pcal}) \setminus F}\|_{\reps X} \\
  & \le \|f^* \chi_{(0, \measl{E'_n})}\|_{\reps X} + \| f^* \chi_{[\meas{B}, \meas{\Pcal})}\|_{\reps X}
  = \|u \chi_{G_n}\|_X + \| u \chi_{\Pcal \setminus B}\|_{X} < \eps.
  \qedhere
\end{align*}
\end{proof}
%
%
\begin{df}
Given $p\in[1, \infty)$ and a quasi-concave function $\phi$ that is constant on $(\meas{\Pcal}, \infty)$, we define the \emph{Marcinkiewicz-type spaces} $M^p_\phi$ and $M^p_{\phi,\loc}$ by their norms
\begin{align*}
  \|u\|_{M^p_\phi} = \sup_{t>0} M_pu^*(t) \phi(t) \quad \mbox{and}\quad
  \|u\|_{M^p_{\phi, \loc}} = \sup_{0<t< 1} M_pu^*(t) \phi(t).
\end{align*}
If $X$ is an \ri space whose fundamental function is $\phi$, then we write $M^p(X)$ and $M^p_\loc(X)$ instead of $M^p_\phi$ and $M^p_{\phi, \loc}$, respectively.
\end{df}
It is easy to verify that $M^p_\phi \emb M^p_{\phi, \loc} \emb L^p_\fm \emb L^p_\loc$. If $\meas{\Pcal}<\infty$, then $M^p_\phi$ and $M^p_{\phi, \loc}$ coincide. On the other hand, if $\meas{\Pcal}=\infty$, then $M^p_\phi$ is non-trivial if and only if $\phi(t)/t^{1/p}$ is bounded for $t>1$. Observe also that $M^1(X) = M(X)$ by definition. The fundamental function of $M^p_\phi$ dominates $\phi$, whereas it is equal to $\phi$ if and only if $\phi^p$ is quasi-concave. The function $\psi(t)$ defined by \eqref{eq:quasiconv_dom} below equals the fundamental function of $M^p_{\phi, \loc}$ for $t\le 1$.
%
%
Having introduced various rather wide classes of function spaces, we can revisit the question of density of bounded functions, providing several examples where the density fails since the function norm is not absolutely continuous.

The following example shows that bounded functions are not dense in the Marcinkiewicz spaces that lie locally strictly between $L^\infty$ and $L^1$.
%
%
\begin{exa}
Let $X=M_\phi$, where $\phi$ is a quasi-concave function that satisfies
\[
  \lim_{t\to 0\limplus} \phi(t) = 0, \quad \mbox{and} \quad \lim_{t\to 0\limplus} \frac{\phi(t)}{t} = \infty.
\]
Roughly speaking, these conditions on $\phi$ say that $L^\infty \subsetneq X_\fm$ and $X \subsetneq L^1_\fm$. By~\cite[Proposition II.5.10]{BenSha}, there exists a concave function $\psi(t) \approx t/\phi(t)$, $t>0$, since the latter is quasi-concave, which follows from quasi-concavity of $\phi$. Moreover, $\psi$ is increasing and absolutely continuous. Let now $u\ge 0$ be chosen such that $u^*(t)$ is the right-continuous representative of the right derivative $\psi'_+(t)$. Since $\psi(t)/t \approx 1/\phi(t) \to \infty$ as $t\to 0\limplus$, we have that $u^*(0\limplus) = \infty$, whence $u$ is not bounded. Now, we will show that $u\in X$, but no sequence of bounded functions converges to $u$ in $X$. Indeed,
\[
  \| u \|_{X} = \sup_{t>0} u^{**}(t) \phi(t) = \sup_{t>0} \frac{\phi(t)}{t} \int_0^t u^*(s) \,ds = \sup_{t>0} \frac{\phi(t) \psi(t)}{t} \approx 1.
\]
Let $f\in X \cap L^\infty$ be non-negative and let $b\coloneq \|f\|_{L^\infty}$. Now, we choose $a>0$ such that $u^*(t) \ge 2b$ for $t\in(0, a)$. Then,
\begin{align*}
  \| u-f \|_{X} & = \sup_{t>0} (u-f)^{**}(t) \phi(t) \ge \sup_{0<t<a} \frac{\phi(t)}{t} \int_0^t (u-f)^*(s)\,ds \\
  & \ge \sup_{0<t<a} \frac{\phi(t)}{t} \int_0^t (u^*(s) - b)\,ds \ge \sup_{0<t<a} \frac{\phi(t)}{t} \frac{\psi(t)}{2} \approx 1.
\end{align*}
\end{exa}
%
%
In the setting of \ri spaces that contain unbounded functions, the absolute continuity (on sets of finite measure) is actually indispensable for the density of bounded functions.
%
%
\begin{lem}
\label{lem:trunc_dense_locAC}
Let $X$ be an \ri space such that $X\setminus L^\infty \neq \emptyset$. Then, the truncations are dense in $X$ if and only if $\| u \chi_{E_k}\|_X \to 0$ as $k\to\infty$ for every $u\in X$ and every decreasing sequence of sets $\{E_k\}_{k=1}^\infty$ such that $\meas{E_k} \to 0$.
\end{lem}
%
%
The latter condition can be understood as \emph{absolute continuity of the norm on sets of finite measure}. If $\meas{\Pcal} < \infty$, then it is equivalent to the absolute continuity of the norm. Otherwise, the absolute continuity is more restrictive since it requires that $\| u \chi_{E_k}\|_X \to 0$ even in the case when $\meas{E_k} = \infty$ for all $k\in\Nbb$ but $\meas{\bigcap_{k=1}^\infty E_k} = 0$, see Example~\ref{exa:AC-vs-ACloc} below.

In order to prove the density of truncations in $X$ under the condition of absolute continuity of the norm on sets of finite measure, we do not really need all the axioms of an \ri space. It would suffice to assume that $X$ is a quasi-Banach function lattice and $X \subset L^p_\fm$ for some $p>0$. Recall that this inclusion with $p=1$ follows by the axiom~\ref{df:BFL.locL1} in the definition of quasi-Banach function spaces (and hence \ri spaces). It is for the converse we make use of the remaining axioms of \ri spaces.
%
%
\begin{proof}
Suppose first that the norm of $X$ is absolutely continuous on sets of finite measure. Let $u\in X$. Recall the notation $\suplev{u}{k} = \{x\in\Pcal: |u(x)|>k\}$. Then, $\meas{\suplev{u}{k}} \to 0$ as $k\to \infty$ by Lemma~\ref{lem:superlevelsets-to-zero} since $X \emb L^1_\fm$ by~\ref{df:BFL.locL1} in the definition of \ri spaces. Let $u_k$ be the truncation of $u$ at the levels $\pm k$. Then, the absolute continuity of the norm of $X$ on sets of finite measure implies that
\[
  \| u - u_k\|_X =\| (|u|-k) \chi_{\suplev{u}{k}}\|_X \le \| u \chi_{\suplev{u}{k}} \|_X \to 0\quad\mbox{as }k\to\infty,
\]
which finishes the proof of the sufficiency.

Suppose next that the norm of $X$ is not absolutely continuous on sets of finite measure, i.e., there exists $v\in X$ and a decreasing sequence of sets $\{E_k\}_{k=1}^\infty$ with $\meas{E_k} \to 0$ such that $\| v \chi_{E_k}\|_X \to a > 0$ as $k\to \infty$. We also have that $\phi(0\limplus) = 0$ since $X$ contains unbounded functions. By passing to a subsequence if needed, we may assume that $\phi(\meas{E_k}) < k^{-2}$. Let $A_k = E_k \setminus E_{k+1}$ for $k\in \Nbb$. Set
\[
  \tilde{u} = \sum_{k=1}^\infty \frac{\chi_{A_k}}{k^{3/2} \phi(\meas{E_k})} \quad\mbox{and}\quad u=\tilde{u}+|v|.
\]
Thus, $u\in X$ while $\tilde{u} > \sqrt{k}$ on $E_k$. Let now $u_n = \min\{u, n\}$ for $n\in\Nbb$. Then, $u - u_n \ge |v|$ on $E_k$ whenever $k \ge n^2$, which yields $\|u-u_n\|_X \ge \| v \chi_{E_{n^2}} \|_X \ge a > 0$ for every $n\in\Nbb$.
\end{proof}
%
%
The next example illustrates that absolute continuity on sets of finite measure really is a more general notion than absolute continuity of the quasi-norm.
%
%
\begin{exa}
\label{exa:AC-vs-ACloc}
The norm of the space $X = (L^1 + L^\infty)(\Pcal)$, where $\meas{\Pcal}=\infty$, is given as $\|u\|_X = \|u^* \chi_{(0,1)}\|_{L^1(\Rbb^+)}$ (cf.\@~\cite[Theorem II.6.4]{BenSha}). It is not absolutely continuous, but merely absolutely continuous on sets of finite measure.

Indeed, let $u \equiv 1 \in X$ and let $E_k= \Pcal \setminus kB$ for $k\in\Nbb$, where $B \subset \Pcal$ is a ball. Then, $\bigcap_{k=1}^\infty E_k = \emptyset$, but $(u\chi_{E_k})^* \equiv 1$, whence $\|u\chi_{E_k}\|_X = 1$ for all $k$. If we however have a decreasing sequence of sets $F_k$ with $\meas{F_k} \to 0$, then $(u \chi_{F_k})^* \chi_{(0,1)} = (u \chi_{F_k})^*$ whenever $\meas{F_k} < 1$, which yields that $\|u \chi_{F_k}\|_X = \|u \chi_{F_k}\|_{L^1} \to 0$ as $k\to\infty$ by the dominated convergence theorem.

Consequently, the truncations are dense in $X$ by Lemma~\ref{lem:trunc_dense_locAC}, whereas this conclusion cannot be drawn from Lemma~\ref{lem:trunc_dense_AC}.
\end{exa}
%
%
We may modify the proof of Corollary~\ref{cor:bdd-dense-in-N1X} similarly as in Lemma~\ref{lem:trunc_dense_locAC} to see that the absolute continuity on sets of finite measure suffices for the density of truncations in $\NX$ provided that $X \subset L^p_\fm$ for some $p>0$.

The Marcinkiewicz spaces, whose norms in general lack the absolute continuity (also on sets of finite measure), will provide us with a setting where we can construct an unbounded Newtonian function whose truncations lie far away from the function. The situation here is somewhat more involved since not only $X$, but also $\NX$ has to contain unbounded functions.
%
%
\begin{exa}
\label{exa:trunc_not_dense}
Let $\phi \in \Ccal^1(\Rbb^+)$ be a quasi-concave function that satisfies
\[
  \lim_{t\to 0\limplus} \phi(t) = 0 \quad \mbox{and} \quad \lim_{t\to 0\limplus} \frac{\phi(t)}{t} = \infty.
\]
Suppose that there are $q>p>1$ such that $\phi(t)^p/t$ is a decreasing function for $t> 0$ whereas $\phi(t)^q/t$ is increasing. In particular, these conditions are satisfied if $\phi^p$ is (quasi)concave whereas $\phi^q$ is convex. The maximal operator $M_1: M_\phi^* \to M_\phi^*$ is bounded under these assumptions, which is shown in Lemma~\ref{lem:fi-quasiconc_weak2weak-glob} below. In view of the Herz--Riesz inequality (Proposition~\ref{pro:herz} below), the Marcinkiewicz spaces $M_\phi$ and $M_\phi^*$ coincide with equivalent (quasi)norms. If, for example, $\phi(t) = t^{1/\alpha}$ with $\alpha > 1$, then we may choose any $p\in (1, \alpha]$ and $q \in (\alpha, \infty)$, and $M_\phi = M^*_\phi = L^{\alpha, \infty}$.

Let $X=M_\phi$ over $\Rbb^n$ endowed with the Euclidean metric and the $n$-dimensional Lebesgue measure, where $n>q$. Then, the truncations are not dense in $\NX$, which is seen by the following argument:

Let $f(t) = t/\phi(t^n)$ for $t>0$. Then, $f$ is decreasing and $f(0) \coloneq f(0\limplus) = \infty$ since
\[
  \lim_{t\to 0\limplus} \frac{t}{\phi(t^n)} = \lim_{t\to 0\limplus} \frac{t^{1/n}}{\phi(t)} = \biggl( \lim_{t\to 0\limplus} \frac{t}{\phi(t)^n} \biggr)^{1/n} = \biggl( \lim_{t\to 0\limplus} \frac{t}{\phi(t)^q} \cdot \frac{1}{\phi(t)^{n-q}}\biggr)^{1/n} = \infty.
\]
Furthermore, $|f'(t)| \approx 1/\phi(t^n)$ since the inequality
\begin{equation}
  \label{eq:phi'-phit/t}
  \frac{\phi(t)}{qt} \le \phi'(t) \le \frac{\phi(t)}{t}\quad\mbox{ for all $t>0$,}
\end{equation}
which holds due to the monotonicity of $\phi(t)/t$ and $\phi(t)^q/t$, leads to
\[
   \frac{-1}{\phi(t^n)} \approx \frac{1-n}{\phi(t^n)} \le \frac{1}{\phi(t^n)} - \frac{n t^n \phi'(t^n)}{\phi(t^n)^2} \le \frac{q-n}{q} \cdot \frac{1}{\phi(t^n)} \approx \frac{-1}{\phi(t^n)} \,.
\]
Let $u(x) = (f(|x|) - f(1))^+$ for $x\in \Rbb^n$. Then, $|\nabla u(x)| = |f'(|x|)|$ for $|x| < 1$. Similarly as in~\cite[Proposition 1.14]{BjoBjo}, we see that $g\coloneq|\nabla u| \chi_{B(0,1)}$ is an upper gradient of $u$. Hence, we can estimate $\|u\|_\NX = \|u\|_X + \|g\|_X \approx \|u\|_{M_\phi^*} + \|g\|_{M_\phi^*}$ while
\begin{align*}
  \|u\|_{M_\phi^*} & = \sup_{t>0} u^*(t) \phi(t) = \sup_{0<r<1} (f(r) - f(1)) \phi(\omega_n r^n) \lesssim \sup_{0<r<1} f(r)\phi(r^n) = 1, \\
  \|g\|_{M_\phi^*} & = \sup_{0<t<\omega_n} (\nabla u)^*(t) \phi(t) = \sup_{0<r<1} |f'(r)| \phi(\omega_n r^n) \approx \sup_{0<r<1} |f'(r)|\phi(r^n) \approx 1,
\end{align*}
where $\omega_n$ is the measure of the $n$-dimensional unit ball. Therefore, $u\in\NX$. We can see that $|\nabla u(x)|$ is also a minimal $X$-weak upper gradient of $u$ by following the argument of~\cite[Proposition A.3]{BjoBjo}, where we replace the representation formula~\cite[Theorem 2.51]{BjoBjo} by~\cite[Theorem~4.10]{Mal2} with $\varphi(t)=t$. Note that the function $\varphi$ in~\cite[Theorem~4.10]{Mal2} is unrelated to $\phi$.
  Let now $u_k(x) = \min\{ u(x), k\}$ for $x\in\Rbb^n$, where $k\in \Nbb$. Then, $g_k \coloneq |\nabla u| \chi_{B(0,r_k)}$, with $r_k = f^{-1}(k + f(1))$, is a minimal $X$-weak upper gradient of $u-u_k$. This provides us with the estimate
\begin{align*}
  \|u-u_k\|_\NX & = \|u-u_k\|_X + \|g_k\|_X \gtrsim \|g_k\|_{M_\phi^*}  = \sup_{0<t<\omega_n r_k^n} (\nabla u)^*(t) \phi(t) \\
  & = \sup_{0<r<r_k} |f'(r)| \phi(\omega_n r^n)
  \approx \sup_{0<r<r_k} |f'(r)|\phi(r^n) \approx 1,
\end{align*}
whence $u \in \NX$ cannot be approximated in $\NX$ by its truncations.
\end{exa}
%
%
%
%
\section{Weak type boundedness of the maximal operator}
\label{sec:weaktype}
The general main result for $p$-Poincar\'e spaces, Theorem~\ref{thm:general_Lip-dens}, relies on the fact that the maximal function $M_pg$ fulfills the weak estimate $\smash{\bigl\|\sigma \chi_{\suplevp{p}{g}{\sigma}}\bigr\|_X} \to 0$ as $\sigma\to\infty$ whenever $g\in X$. Recall that $\suplevp{p}{g}{\sigma}$ denotes the superlevel set of $M_pg$ with level~$\sigma$, i.e., $\suplevp{p}{g}{\sigma} = \{x\in\Pcal: M_p g(x) > \sigma\}$. In this section, we will show that this condition is satisfied in the setting of \ri spaces if $M_p: X \to M^*(X)_\fm$ is bounded. Furthermore, we will establish various sufficient conditions on $X$, and in particular on its fundamental function, that guarantee such boundedness of $M_p$.

The Herz--Riesz inequality is a crucial tool for studying the maximal operators on \ri spaces. It allows us to compare the elementary maximal function, i.e., the maximal function of the rearrangement, with the rearrangement of the maximal function.
%
%
\begin{pro}[Herz--Riesz inequality]
\label{pro:herz}
There are constants $c,c'>0$ such that
\[
  c (M_1 u)^*(t) \le u^{**}(t) \le c' (M_1 u)^*(t),\quad t\in (0, \meas{\Pcal}),
\]
whenever $u\in \Mcal(\Pcal, \mu)$.
\end{pro}
%
%
F.~Riesz~\cite{Rie} used the rising sun lemma to prove the estimate $(M_1u)^* \lesssim u^{**}$ for functions defined on the interval $[0,1]$ in 1932. The inequality in (unweighted) $\Rbb^n$ follows from Wiener~\cite{Wie}. The converse estimate was established much later (1968) and is attributable to Herz~\cite{Her} in one dimension, and to Bennett and Sharpley~\cite{BenSha} in $n$ dimensions. See also Asekritova, Kruglyak, Maligranda, and Persson~\cite{AseKruMalPer}.
%
%
\begin{proof}
If $u \notin L^1_\loc(\Pcal)$, then trivially $(M_1 u)^* = u^{**}\equiv \infty$. We can therefore suppose that $u \in L^1_\loc(\Pcal)$.
The proof of Theorem III.3.8 in Bennett and Sharpley~\cite{BenSha} works verbatim for the left-hand inequality even in the setting of metric measure spaces.

The proof of the right-hand inequality is however somewhat more involved. Let $u \in L^1_\loc(\Pcal)$ and $t>0$ be given. We may suppose that $(M_1 u)^*(t) < \infty$ since the inequality holds trivially otherwise. Let $E = \suplevp{1}{u}{(M_1u)^*(t)}$, i.e., $E=\{x\in\Pcal: M_1 u(x) > (M_1 u)^*(t)\}$. Then, $\meas{E} \le t$, and $E$ is open since $M_1 u$ is lower semicontinuous by~\cite[Lemma 3.12]{BjoBjo}.

Let $\{B_\alpha\}_{\alpha \in I}$ be a Whitney-type covering of $E$ by open balls, i.e., it satisfies:
\begin{enumerate}
  \renewcommand{\theenumi}{(\roman{enumi})}
  \item if $\alpha, \beta \in I$ and $\alpha \neq \beta$, then $B_\alpha \cap B_\beta = \emptyset$,
  \item $E = \bigcup_{\alpha\in I} c_W B_\alpha$,
  \item $4c_W B_\alpha \setminus E \neq \emptyset$ whenever $\alpha \in I$,
\end{enumerate}
where $c_W\ge 1$ is an absolute constant. Existence of such a covering is established, e.g., in Auscher and Bandara~\cite[Theorem 2.3.4]{AusBan}. Since $\Pcal = \spt \mu$ is a Lindel\"{o}f space by~\cite[Proposition 1.6]{BjoBjo}, the index set $I$ may be assumed at most countable. Let now $G = \bigcup_{\alpha\in I} 4c_W B_\alpha$, $v=u \chi_G$, and $w = u-v = u \chi_{\Pcal \setminus G}$. Then, we can estimate
\[
  \|w\|_{L^\infty} \le \| u \chi_{\Pcal \setminus E}\|_{L^\infty} \le \| (M_1u) \chi_{\Pcal \setminus E}\|_{L^\infty} \le (M_1u)^*(t).
\]
For every $\alpha\in I$, there is $x_\alpha \in 4c_W B_\alpha \setminus E$. Hence,
\[
  (M_1u)^*(t) \ge M_1u (x_\alpha) \ge \fint_{4c_WB_\alpha} |u|\,d\mu.
\]
This allows us to estimate
\begin{align*}
  \|v\|_{L^1} & = \int_G |u|\,d\mu \le \sum_{\alpha \in I} \int_{4c_WB_\alpha} |u|\,d\mu \le \sum_{\alpha \in I} (M_1u)^*(t) \meas{4c_WB_\alpha} \\
  & \le \tilde{c} (M_1u)^*(t) \sum_{\alpha \in I}  \meas{B_\alpha} \le \tilde{c} (M_1u)^*(t) \meas{E} \le \tilde{c} t (M_1u)^*(t),
\end{align*}
where $\tilde{c}\ge1$ depends only on the doubling constant of $\mu$ and on $c_W$. Due to the subadditivity of the elementary maximal operator, we obtain that
\begin{align*}
u^{**}(t) & \le v^{**}(t) + w^{**}(t) = \fint_0^t (v^*(s) + w^*(s))\,ds \\
& \le \frac{\|v\|_{L^1}}{t} + \|w\|_{L^\infty} \le (\tilde{c}+1) (M_1u)^*(t).
\qedhere
\end{align*}
\end{proof}
%
%
As a direct consequence of the Herz--Riesz inequality for $M_1$, we can also estimate the rearrangement of $M_p u$ for any $p\in[1, \infty)$.
%
%
\begin{cor}[Herz--Riesz inequality]
\label{cor:herz}
For every $p\in[1,\infty)$, there are $c,c'>0$ such that
\[
  c (M_p u)^*(t) \le M_pu^*(t) \chi_{(0, \meas{\Pcal})}(t) \le c' (M_p u)^*(t),\quad t\in \Rbb^+,
\]
whenever $u\in \Mcal(\Pcal, \mu)$.
\end{cor}
%
%
\begin{proof}
By the definition of the decreasing rearrangement, we have $(M_p u)^*(t) = 0$ whenever $t\ge \meas{\Pcal}$. In view of Proposition~\ref{pro:herz}, we can estimate for $t < \meas{\Pcal}$ that
\begin{align*}
  (M_p u)^*(t) & = ((M_1 |u|^p)^{1/p})^*(t) = (M_1 |u|^p)^*(t)^{1/p} \approx (|u|^p)^{**}(t)^{1/p} \\
  & = M_1 ((|u|^p)^*)(t)^{1/p} = M_1((u^*)^p)(t)^{1/p} = M_p u^* (t).
\qedhere
\end{align*}
\end{proof}
%
%
The following lemma shows that if $M_p$ is a bounded operator from $X$ to $X_\fm$, then it obeys the desired weak type estimate. In this case, we allow $X$ to be a more general function lattice with absolutely continuous norm. Recall that a (sub)linear operator~$T$ is bounded from $X$ to $X_\fm$ if for every $E\subset \Pcal$ of finite measure there is $c_E \ge 0$ such that $\|(T u)\chi_E\|_X \le c_E \|u\|_X$ whenever $u\in X$.
%
%
\begin{lem}
\label{lem:key_est-M_bdd}
Let $X$ be a quasi-Banach function lattice with absolutely continuous quasi-norm. Suppose that $X \subset L^p_\fm$ and that $M_p: X \to X_\fm$ is bounded for some $p\in[1, \infty)$. If $v\in X$, then $\bigl\|\sigma \suplevp{p}{v}{\sigma}\bigr\|_X \to 0$ as $\sigma \to \infty$.
\end{lem}
%
%
\begin{proof}
Since $M_p v\in X_\fm \subset L^p_\fm$, we may use Lemma~\ref{lem:superlevelsets-to-zero} to prove that there is $\sigma_0>0$ such that $\meas{\suplevp{p}{v}{\sigma_0}} < \infty$ and that $\meas{\suplevp{p}{v}{\sigma}} \to 0$ as $\sigma \to \infty$. 

Then, using a Chebyshev-type estimate for $\sigma \ge \sigma_0$ and the boundedness of $M_p$, we obtain that
\[
  \| \sigma\chi_{\suplevp{p}{v}{\sigma}}\|_X \le \| (M_pv) \chi_{\suplevp{p}{v}{\sigma}}\|_X \le \| (M_pv) \chi_{\suplevp{p}{v}{\sigma_0}}\|_X \le c_{\sigma_0} \|v\|_X < \infty.
\]
The absolute continuity of the norm gives that $\|(M_pv) \chi_{\suplevp{p}{v}{\sigma}}\|_X \to 0$ as $\sigma\to\infty$ since $(M_pv) \chi_{\suplevp{p}{v}{\sigma_0}} \in X$ and $\meas{\suplevp{p}{v}{\sigma}} \to 0$. Hence, $\| \sigma \chi_{\suplevp{p}{v}{\sigma}}\|_X \to 0$ as $\sigma\to\infty$.
\end{proof}
%
%
The following example shows that the assumption on absolute continuity of the quasi-norm of $X$ is crucial in the previous lemma and without it the weak type estimate may fail even though $M_p$ is bounded.
\begin{exa}
\label{exa:Mp-bdd_notAC}
The Herz--Riesz inequality (or the Marcinkiewicz interpolation theorem) yields that $M_p: L^{q,s} \to L^{q,s}$ is bounded for all $q\in (p, \infty)$ and $s\in[1, \infty]$. Let us consider $X = L^{q, \infty}(\Rbb^+)$ for arbitrary $q\in (p, \infty)$ and $g(t) = t^{-1/q}$ for $t \in \Rbb^+$. Since $g$ is decreasing, we obtain that $M_p g(t) = \bigl(\fint_0^t s^{-p/q}\,ds\bigr)^{1/p} = c_{p,q} t^{-1/q}$ for every $t\in\Rbb^+$. Hence, $\suplevp{p}{g}{\sigma} = (0, c_{p,q}^q / \sigma^q)$, which gives
\[
  \Bigl\| \sigma \chi_{\suplevp{p}{g}{\sigma}}\Bigr\|_X = \sigma \sup_{t>0} \chi_{\suplevp{p}{g}{\sigma}}^*(t) \phi_X(t) =  \sigma \sup_{0<t<c^q_{p,q}/\sigma^q} t^{1/q} = c_{p,q} > 0
\]
regardless of the value of $\sigma>0$.
\end{exa}
%
%
In the rest of this section, we will be describing weak boundedness (on sets of finite measure) of $M_p$ on \ri spaces, which will be considerably easier as we will apply the Herz--Riesz inequality to reduce the problem and investigate the behavior of the maximal functions on $\Rbb^+$ with $1$-dimensional Lebesgue measure instead.
%
%
\begin{rem}
It follows from the Herz--Riesz inequality that $M_p: X \to M^*(Y)$ is bounded if and only if $M_p: \reps{X} \to M^*(\reps{Y})$ is bounded whenever $X$ and $Y$ are \ri spaces over $(\Pcal, \mu)$ and $p\in[1, \infty)$.

In fact, it can be also shown that $M_p: X \to M^*(Y)_\fm$ is bounded if and only if $M_p: \reps{X} \to M^*(\reps{Y})_\fm$ is bounded. The proof of this statement is however more involved since a uniform correspondence between sets of finite measure in $\Rbb^+$ and in $\Pcal$ needs to be established. In other words, given $f\in \reps{X}$, an equimeasurable $u\in X$ needs to be found so that its level sets have a certain structure, independently of $f$.

Neither of these claims will be used in this paper, whence their proof is omitted.
\end{rem}
Next, we will see that the mere weak boundedness of $M_p$ on sets of finite measure for \ri spaces is sufficient for the desired weak type estimate. Considering the simple example $X=L^p$, we see that weak boundedness of $M_p$ is indeed more general than boundedness of $M_p$, which was required in Lemma~\ref{lem:key_est-M_bdd}. Note that we cannot omit the hypothesis that the norm of $X$ is absolutely continuous since that would lead to invalidity of the claim, which we have already observed in Example~\ref{exa:Mp-bdd_notAC}.
%
%
\begin{lem}
\label{lem:key-Mp-wbdd}
Let $X$ be an \ri space with absolutely continuous norm. Suppose that $M_p: X \to M^*(X)_\fm$ is bounded for some $p\in[1, \infty)$. If $v\in X$, then $\bigl\|\sigma \chi_{\suplevp{p}{v}{\sigma}}\bigr\|_X \to 0$ as $\sigma \to \infty$.
\end{lem}
%
%
\begin{proof}
First, we shall show that $X \emb L^p_\fm$. Let $u\in X$. Since $(M_p u)\chi_A \in M^*(X)$ for every measurable set $A\subset \Pcal$ of finite measure, we have $M_p u <\infty$ a.e.\@ in $\Pcal$. Consequently, $u\chi_B \in L^p$ for every ball $B\subset \Pcal$. Let $E \subset \Pcal$ with $\meas{E}<\infty$. Then, there exists a ball $B_E \subset \Pcal$ such that $\meas{E} \le \meas{B_E} < \infty$. By~\cite[Corollary~II.7.8]{BenSha}, there is a measurable function $\tilde{u} = \tilde{u} \chi_{B_E}$ such that $\tilde{u}^* = (u\chi_E)^*$. By the lattice property~\ref{df:BFL.latticeprop} of $\reps{X}$, we obtain that $\tilde{u} \in X$, whence $\tilde{u} \chi_{B_E} \in L^p$. Now, $\|u\chi_E\|_{L^p} =\|\tilde{u}\chi_{B_E}\|_{L^p} < \infty$ and that is why $u\in L^p_\fm$. Thus, $X \subset L^p_\fm$. In particular, the spaces of functions restricted to $E$ satisfy $X(E) \subset L^p(E)$. These spaces are \ri spaces as well and hence the embedding is continuous by~\cite[Theorem~I.1.8]{BenSha}, i.e., $\|u\chi_E\|_{L^p(\Pcal)} = \|u\|_{L^p(E)} \le c_E \|u\|_{X(E)} = c_E \|u\chi_E\|_{X}$. Therefore, $X \emb L^p_\fm$.

Next, we will show that $\meas{\suplevp{p}{v}{\sigma}} \to 0$ as $\sigma\to\infty$. Suppose on the contrary that $\meas{\suplevp{p}{v}{\sigma}} > a >0$ for every $\sigma>0$. Then, there exist pairwise disjoint sets $F_k \subset \suplevp{p}{v}{k\log k}$ that satisfy $\meas{F_k} = a /(k^2 + k)$, $k\in\Nbb$. Therefore, $M_p v \ge \sum_{k=1}^\infty (k \log k) \chi_{F_k}$. We also have $M^*(X) \emb L^{p,\infty}_\fm$, which follows from the inequality $\phi_{L^p} (t) \le c_b \phi_X(t)$ for $t\in (0, b)$ with arbitrary $b>0$, which in turn follows from the embedding $X \emb L^p_\fm$. Let $F = \bigcup_{k=1}^\infty F_k$. Then, $\meas{F} = a < \infty$ and
\begin{align*}
  \infty &= \sup_{k\ge1} \biggl(\frac{a}{k}\biggr)^{1/p} k\log k = \sup_{t>0} t^{1/p} \sum_{k=1}^\infty (k\log k) \chi_{[a/(k+1), a/k)}(t) \\ 
  & = \biggl\| \sum_{k=1}^\infty (k\log k)\chi_{F_k} \biggr\|_{L^{p,\infty}}  \le c_F \|(M_pv) \chi_F\|_{M^*(X)} \le c_F' \|v\|_X < \infty,
\end{align*}
which is a contradiction and hence $\meas{\suplevp{p}{v}{\sigma}} \to 0$ as $\sigma \to \infty$.

Let $\sigma_0>0$ be chosen such that $\meas{\suplevp{p}{v}{\sigma_0}} < \infty$. For $\sigma > \sigma_0$, we define $v_\sigma = v \chi_{\suplev{v}{\sigma/2}}$. Then, $M_p v \le M_p v_\sigma + \sigma/2$. In particular, $M_p v_\sigma > \sigma/2$ on $\suplevp{p}{v}{\sigma}$. Using a Chebyshev-type estimate and the boundedness of $M_p: X \to M^*(X)_\fm$, we see that
\begin{align*}
   \bigl\|\sigma \chi_{\suplevp{p}{v}{\sigma}}\bigr\|_{M^*(X)}  \lesssim \bigl\|(M_p v_\sigma) \chi_{\suplevp{p}{v}{\sigma_0}}\bigr\|_{M^*(X)} 
	  \le c_{\sigma_0} \|v_\sigma\|_{X} = c_{\sigma_0} \|v \chi_{\suplev{v}{\sigma/2}}\|_{X} \to 0
\end{align*}
as $\sigma\to\infty$ since the norm of $X$ is absolutely continuous and $\meas{\bigcap_{\sigma>0} \suplev{v}{\sigma}}=0$.
\end{proof}
%
%
Due to the definitions of the Marcinkiewicz-type spaces $M^p(X)$, $M^p_\loc(X)$, and $M^*(X)$, we will  obtain that $M_p$ is weakly bounded on $M^p(X)$, and on $M^p_\loc(X)$ on sets of finite measure. Consequently, $M_1$ is weakly bounded on all \ri spaces. In view of Lemma~\ref{lem:key-Mp-wbdd}, we obtain the desired weak estimate, which is needed to conclude the density of Lipschitz functions in $\NX$ on $1$-Poincar\'e spaces using Theorem~\ref{thm:general_Lip-dens}, whenever $X$ is an \ri space.
\begin{pro}
\label{pro:Mp-wbdd}
Let $X$ be an \ri space. Then, $M_p: M^p(X) \to M^*(X)$ is bounded for all $p\in[1, \infty)$. In particular, $M_1: X\to M^*(X)$ is bounded.

Furthermore, $M_p: M^p_\loc(X)\to M^*(X)_\fm$ is bounded. If $X \emb M^p_\loc(X)$, then in particular $M_p: X\to M^*(X)_\fm$ is bounded.
\end{pro}
%
%
\begin{proof}
Let $u\in M^p(X)$. Then, the Herz--Riesz inequality yields
\[
  \|M_p u\|_{M^*(X)} = \sup_{t>0} (M_pu)^*(t) \phi(t) \approx \sup_{t>0} M_p u^*(t) \phi(t) = \|u\|_{M^p(X)}.
\]
The restriction $M_1: X \to M^*(X)$ is bounded since $X \emb M(X) = M^1(X)$ by \eqref{eq:LambdaX-X-MX-embed}.

Let now $u\in M^p_\loc(X)$. Let $E \subset \Pcal$ with $\meas{E}<\infty$. With appeal to the Herz--Riesz inequality, we obtain
\begin{align*}
  \|(M_p u) \chi_E\|_{M^*(X)}  &= \sup_{t>0} ((M_pu) \chi_E)^*(t) \phi(t) \le \sup_{0<t<\meas{E}} (M_pu)^*(t) \phi(t)\\
   & \approx \sup_{0<t<\meas{E}} M_pu^*(t) \phi(t) \le \sup_{0<t<1+\meas{E}} M_pu^*(t) \phi(t).
\end{align*}
By the quasi-concavity of $\phi$,  we have that $\phi(t)/(1+\meas{E}) \le \phi\bigl(t/(1+\meas{E})\bigr)$. Monotonicity of $M_pu^*$ then gives us that
\begin{align*}
  \sup_{0<t<1+\meas{E}} M_pu^*(t) \phi(t) & \le (1+\meas{E}) \sup_{0<t<1+\meas{E}} M_pu^*\Biggl(\frac{t}{1+\meas{E}}\Biggr) \phi\Biggl(\frac{t}{ 1+\meas{E}}\Biggr) \\
  & = (1+\meas{E}) \|u\|_{M^p_\loc(X)}\,.
\end{align*}
We have thus shown that $\|(M_p u) \chi_E\|_{M^*(X)} \le c_E \|u\|_{M^p_\loc(X)}$. The boundedness of $M_p: X \to M^*(X)_\fm$ immediately follows provided that $X \emb M^p_\loc(X)$.
\end{proof}
%
%
Let us now take a look at an example that illustrates the difference in strength of the claims for $M^p(X)$ and $M^p_\loc(X)$ in the previous proposition.
\begin{exa}
\label{exa:Mp-vs-Mploc}
Suppose that $1 \le q < p \le s < \infty$ and let $X = (L^q \cap L^s)(\Pcal)$ with a norm given by $\|u\|_X = \max \{ \|u\|_{L^q(\Pcal)}, \|u\|_{L^s(\Pcal)}\}$ for $u\in\Mcal(\Pcal, \mu)$. Then, $X$ has fundamental function $\phi(t) = \max\{t^{1/q}, t^{1/s}\}$ for $t\in(0, \meas{\Pcal})$. The H\"older inequality yields that
\begin{align*}
  \|u\|_{M_\loc^p(X)} & = \sup_{0<t<1} M_pu^*(t)\phi(t) = \sup_{0<t<1} \biggl( \fint_0^t u^*(\tau)^p\,d\tau\biggr)^{1/p} t^{1/s}  \\
	& \le \sup_{0<t<1} \biggl( \fint_0^t u^*(\tau)^s\,d\tau\biggr)^{1/s} t^{1/s} = \|u^* \chi_{(0,1)}\|_{L^s(\Rbb^+)} \le \|u\|_{L^s(\Pcal)} \le \|u\|_X.
\end{align*}
Hence, $X \emb M^p_\loc(X)$. If $\meas{\Pcal}<\infty$, then $M^p(X) = M^p_\loc(X)$, which gives us that $X \emb M^p(X)$ in this case. Suppose instead that $\meas{\Pcal}=\infty$. Then,
\begin{align*}
  \sup_{t>1} M_pu^*(t)\phi(t)  = \sup_{t>1} \biggl( \fint_0^t u^*(\tau)^p\,d\tau\biggr)^{1/p} t^{1/q} = \sup_{t>1} t^{1/q-1/p} \|u^* \chi_{(0,t)}\|_{L^p(\Rbb^+)} = \infty
\end{align*}
unless $\|u\|_{L^p} = 0$. Thus, $M^p(X) = \{ u \in \Mcal(\Pcal, \mu): u=0\mbox{ a.e.}\}$.
Therefore, boundedness of the maximal operator $M_p$ on $M^p(X)$ does not provide us with any useful information on boundedness of $M_p$ on $X$ if $\meas\Pcal = \infty$.

On the other hand, the previous proposition yields that $M_p: X \to M^*(X)_\fm$ (regardless of the measure of $\Pcal$), which suffices in Lemma~\ref{lem:key-Mp-wbdd} to obtain the weak type estimate that is used in Theorem~\ref{thm:general_Lip-dens} to prove density of Lipschitz functions in $\NX$ on $p$-Poincar\'e spaces. See also Example~\ref{exa:Lip-dens-LqcapLs} below, where the case when $0<q<1$ is discussed as well.
\end{exa}
%
%
The following technical lemma helps us find a function $\psi$, which dominates a given quasi-concave $\phi$, so that $\psi^p$ is quasi-concave on $[0, 1)$. Moreover, $M^p_{\psi} = M^p_{\phi}$ and  $M^p_{\psi, \loc} = M^p_{\phi, \loc}$. In fact, $\psi$ equals the fundamental function of $M^p_{\phi, \loc}$ on $[0,1]$.
%
%
\begin{lem}
\label{lem:quasiconc_maj}
Let $\phi$ be a quasi-concave function that is constant on $(\meas{\Pcal}, \infty)$ and let $p\in[1, \infty)$. Define $\psi$ by  
\begin{equation}
  \label{eq:quasiconv_dom}
  \psi(t) =
    \begin{cases}
      0&\mbox{for }t=0,\\
      \displaystyle{t^{1/p} \sup_{t \le s \le 1} \frac{\phi(s)}{s^{1/p}}}& \mbox{for }0< t\le1,\\
      \phi(t)&\mbox{for }t \ge 1.
    \end{cases}
\end{equation}
Then, $\psi^p$ is quasi-concave on $[0,1)$. Moreover, $M^p_\psi = M^p_\phi$ and $M^p_{\psi, \loc} = M^p_{\phi, \loc}$ with equality of the respective (quasi)norms.
\end{lem}
%
%
Observe that if $\phi^p$ is quasi-concave on $[0,1)$, then the supremum is attained for $s=t$ whence $\phi(t) = \psi(t)$ for every $t\ge0$.
%
%
\begin{proof}
Let us first show that $\psi^p$ is indeed quasi-concave on $[0,1)$. It follows directly from its definition that $\psi(t)^p/t$ is decreasing for $t\in(0,1)$. Let now $0 < t_1 < t_2 \le 1$. Due to the continuity of $\phi(s)^p/s$ on $[t_1, 1]$, the suprema defining $\psi(t_1)$ and $\psi(t_2)$ are attained at $s_1\in[t_1, 1]$ and $s_2\in[t_2, 1]$, respectively. We distinguish two cases. If $s_1\ge t_2$, then we may choose $s_2=s_1$, whence $\psi(t_2)^p/\psi(t_1)^p = t_2/t_1 > 1$. Therefore, $\psi(t_2)^p>\psi(t_1)^p$. Now, suppose instead that $s_1\in[t_1, t_2)$. Then,
\[
  0<\psi(t_1)^p = \frac{t_1 \phi(s_1)^p}{s_1} \le \phi(t_1)\phi(s_1)^{p-1} \le \phi(t_2)^p = \frac{t_2 \phi(t_2)^p}{t_2} \le \frac{t_2 \phi(s_2)^p}{s_2} = \psi(t_2)^p.
\]

The local norms are equal since
\begin{align*}
  \| u \|_{M^p_{\psi,\loc}} & = \sup_{0<t< 1} M_p u^*(t)\psi(t)
   = \sup_{0<t< 1} \biggl( \sup_{t < x < 1} \frac{\phi(x)^p}{x} \int_0^t u^*(s)^p\,ds \biggr)^{1/p}  \\
   & = \biggl( \sup_{0 < x < 1} \frac{\phi(x)^p}{x} \int_0^x u^*(s)^p\,ds \biggr)^{1/p} 
   = \sup_{0<x< 1} M_p u^*(x)\phi(x) = \| u \|_{M^p_{\phi,\loc}}.
\end{align*}
For the global norms, we have
\begin{align*}
  \| u \|_{M^p_\psi} & = \sup_{t>0} M_p u^*(t)\psi(t) = \max \biggl\{ \sup_{0<t< 1} M_p u^*(t)\psi(t), \sup_{t \ge 1} M_p u^*(t)\psi(t)\biggr\} \\
  & = \max \biggl\{ \sup_{0<t< 1} M_p u^*(t)\phi(t), \sup_{t\ge 1} M_p u^*(t)\phi(t)\biggr\} = \| u \|_{M^p_\phi}\,.
\qedhere
\end{align*}
\end{proof}
%
%
Since the space $M^p_\loc(X)$ is defined using $M_p$, the question of when $X \emb M^p_\loc(X)$ is in principle equivalent to determining whether $M_p$ is weakly bounded on $X$ on sets of finite measure. We can however find classical Lorentz spaces that are embedded in $M^p_\loc(X)$. Hence, we may use Lemma~\ref{lem:key-Mp-wbdd} to obtain the desired weak estimate whenever we show that $X$ is embedded in such a Lorentz space.
%
%
\begin{lem}
\label{lem:Lambdap_MpX}
Let $X$ be an \ri space with fundamental function $\phi$ and let $\psi$ be defined by \eqref{eq:quasiconv_dom}. Then, $\Lambda^p_\psi \emb M^p_\loc(X)$.
In particular, if $X \emb \Lambda^q_{\psi, \fm}\fcrim$ for some $q\le p$, then $X \emb M^p_\loc(X)$.
\end{lem}
%
%
\begin{proof}
Lemma~\ref{lem:quasiconc_maj} and the Herz--Riesz inequality yield
\begin{align*}
  \| u \|_{M^p_{\loc}(X)} & = \| u \|_{M^p_{\phi,\loc}} = \| u \|_{M^p_{\psi,\loc}}
	= \sup_{0<t<1} M_p u^*(t) \psi(t) \\
  &= \sup_{0<t<1} \frac{\psi(t)}{t^{1/p}} \biggl(\int_0^t u^*(s)^p\,ds\biggr)^{1/p} \le \biggl(\int_0^1 u^*(s)^p\psi(s)^p \frac{ds}{s}\biggr)^{1/p} = \|u \chi_G\|_{\Lambda^p_\psi},
\end{align*}
where $G \in \suplevr{u}{\min\{1, \meas{\Pcal}\}}$.
Finally, we obtain that 
\[
  \|u \chi_G\|_{\Lambda^p_\psi} \le c_{p,q} \|u \chi_G\|_{\Lambda^q_\psi} \le c_{p,q}c_{\meas{G}} \|u\|_X
\]
by the embedding of the classical Lorentz spaces (Lemma~\ref{lem:lorentz-embedding}).
\end{proof}
%
%
\begin{rem}
If $\psi^p$ is the highest power of $\psi$ that is quasi-concave on $(0,1)$, then the inclusion $\smash{\Lambda^p_{\psi,\fm}} \emb M^p_\loc(X)$ is rather sharp, in particular when comparing the classical Lorentz spaces with $M^p_\loc(X)$. For example, let $X=L^{q,s}$ with $q \in (1, p)$ and $s\in[1, \infty]$ and suppose that $\meas\Pcal = \infty$. Then, $\psi(t) = \max\{t^{1/p}, t^{1/q}\}$. Similarly as in Example~\ref{exa:Mp-vs-Mploc}, we obtain that $M^p(X) = M^p_\psi = \{u\in\Mcal(\Pcal, \mu): u=0\mbox{ a.e.}\}$. On the other hand, $M^p_\loc(X) = L^p_\fm(\Pcal) = \Lambda^p_{\psi,\fm}(\Pcal)$ since
\begin{align}
\notag
  \|u \|_{M^p_\loc(X)} & = \|u \|_{M^p_{\psi,\loc}} = \sup_{0<t<1} \biggl(\fint_0^t u^*(\tau)\,d\tau\biggr)^{1/p}\psi(t) \\ & =  \sup_{0<t<1} \biggl(\int_0^t u^*(\tau)\,d\tau\biggr)^{1/p} 
	 = \|u^* \chi_{(0,1)}\|_{L^p(\Rbb^+)} = \|u^* \chi_{(0,1)}\|_{\Lambda^p_\psi(\Rbb^+)}\,.
  \label{eq:Mploc-vs-Lambdapsi}
\end{align}

If $X=L^{p, s}$ with $s\in[1, \infty]$, then $\psi(t) = \phi(t) = t^{1/p}$ and a calculation analogous to~\eqref{eq:Mploc-vs-Lambdapsi} yields that $M^p(X) = L^p = \Lambda^p_\psi$, while $M^p_\loc(X) = L^p_\fm = \Lambda^p_{\psi,\fm}$.

If $\phi^q$ is quasi-concave for some $q>p$, then $\psi = \phi$ and Lemma~\ref{lem:fi-quasiconc_weak2weak-glob} below yields that $M_p: M^*(\reps X) \to M^*(\reps X)$ is bounded. Then, $X\emb M^p(X)$ as the boundedness of $M_p$ and \eqref{eq:LambdaX-X-MX-embed} give that $\|u\|_{M^p(X)} = \|M_p u^*\|_{M^*(\reps{X})} \lesssim \|u^*\|_{M^*(\reps{X})} \le \|u^*\|_\reps{X} = \|u\|_X$.
\end{rem}
%
%
Without any additional information on the structure of the norm of an \ri space $X$, it is nearly impossible to describe when $X \emb M^p_\loc(X)$. We can however establish rather general characterizations of the boundedness of $M_p$ on sets of finite measure when both the source space and the target space are weak Marcinkiewicz spaces. In other words, we will study when $M^*(X) \emb M^p_\loc(X)$ using the properties of the fundamental function, which proves helpful since $X \emb M^*(X)$ as seen in \eqref{eq:LambdaX-X-MX-embed}.
%
%
\begin{pro}
\label{pro:Mp-weak2weak}
For every $p\in[1, \infty)$, the mapping $M_p: M^*(X) \to M^*(X)_\fm$ is bounded if and only if
\begin{equation}
  \label{eq:phi-p-avg_bdd}
  \sup_{0<t<1} \phi(t)^p \fint_0^t \frac{ds}{\phi(s)^p} < \infty.
\end{equation}
\end{pro}
%
%
\begin{proof}
Suppose first that $M_p: M^*(X) \to M^*(X)_\fm$ is bounded. By~\cite[Corollary~II.7.8]{BenSha}, there exists $v \in M^*(X)$ such that $v^* = 1/\phi$. Let $A \in \suplevr{M_p v}{1}$. Then, the Herz--Riesz inequality yields that
\begin{multline*}
  \sup_{0<t<1} \biggl( \phi(t)^p \fint_0^t \frac{ds}{\phi(s)^p}\biggr)^{1/p} 
	 =	\sup_{0<t<1} M_p \frac{1}{\phi} (t) \phi(t) 
	= \sup_{0<t<1} M_p v^*(t) \phi(t) \\
	\approx \sup_{0<t<1} (M_p v)^*(t) \phi(t) 
	= \| (M_p v) \chi_{A}\|_{M^*(X)} 
	\lesssim \| v \|_{M^*(X)} = \biggl\| \frac{1}{\phi} \biggr\|_{M^*(X)} 
	= 1.
\end{multline*}

For the proof of the converse, suppose that the expression $\phi(t)^p \fint_0^t \phi(s)^{-p}ds$ is bounded for $t\in(0,1)$. Due to the continuity and monotonicity of $\phi$, it follows that it is bounded for $t\in(0, b)$ for every $b \in \Rbb^+$. Let  $u\in M^*(X)$ and $E\subset \Pcal$ with $b \coloneq \meas{E} < \infty$. Using the Herz--Riesz inequality, we obtain that
\[
 \|(M_p u) \chi_E\|_{M^*(X)}
 \le \sup_{0<t<b} (M_pu)^*(t) \phi(t) \approx \sup_{0<t<b} M_pu^*(t) \phi(t).
\]
We can also estimate
\begin{align*}
  \sup_{0<t<b} M_pu^*(t) \phi(t) & = \sup_{0<t<b} \phi(t) \biggl(\fint_0^t u^*(s)^p \phi(s)^p \frac{ds}{\phi(s)^p}\biggr)^{1/p} 
  \\ 
	&\le \sup_{0<t<b} \biggl( \phi(t)^p \fint_0^t \frac{ds}{\phi(s)^p}\biggr)^{1/p} \sup_{s>0} u^*(s) \phi(s) = c_E \|u\|_{M^*(X)}.
\end{align*}
Hence, $\|(M_p u) \chi_E\|_{M^*(X)} \le c_{E} \|u\|_{M^*(X)}$ as desired.
\end{proof}
%
%
The following characterization of the global boundedness of $M_p$ on the weak space $M^*(X)$ can be proven along the same lines. Thus, the proof is omitted.
\begin{pro}
\label{pro:Mp-weak2weak-glob}
For every $p\in[1, \infty)$, the mapping $M_p: M^*(X) \to M^*(X)$ is bounded if and only if
\begin{equation}
  \label{eq:phi-p-avg_bdd-glob}
  \sup_{t>0} \phi(t)^p \fint_0^t \frac{ds}{\phi(s)^p} < \infty.
\end{equation}
\end{pro}
%
%
We will see in the next lemma that \eqref{eq:phi-p-avg_bdd} is satisfied provided that a certain power of $\phi$ is quasi-concave near zero. In particular, it follows that $M_p: X \to M^*(X)_\fm$ is bounded and Lemma~\ref{lem:key-Mp-wbdd} can be applied to verify the hypotheses of Theorem~\ref{thm:general_Lip-dens}, the general result on density of Lipschitz functions in $\NX$.
\begin{lem}
\label{lem:fi-quasiconc_weak2weak}
Let $\phi: [0,\infty) \to [0, \infty)$ be an increasing function with $\phi(t) = 0$ if and only if $t=0$. Let $q>p\ge1$. If $\phi^q$ is concave or quasi-concave on $[0, \delta)$ for some $\delta>0$, then the condition \eqref{eq:phi-p-avg_bdd} is fulfilled.
\end{lem}
%
%
\begin{proof}
If $\phi^q$ is concave on $[0, \delta)$, then it is quasi-concave there by \eqref{eq:concave-quasiconcave}.

Suppose now that $\phi^q$ is quasi-concave on $[0, \delta)$. For $t\in (0, \delta)$, we obtain
\[
  \phi(t)^p \fint_0^t \frac{ds}{\phi(s)^p} = \biggl( \frac{\phi(t)^q}{t}\biggr)^{p/q} t^{p/q} \fint_0^t \frac{ds}{\phi(s)^p} \le t^{p/q} \fint_0^t \frac{ds}{s^{p/q}} = \frac{q}{q-p}\,.
\]
If $\delta<1$, then $\phi^p, \phi^{-p} \in L^\infty([\delta, 1])$, giving the claimed boundedness on $(0,1)$.
\end{proof}
We can also obtain a global result by a minor tweak of the previous argument.
\begin{lem}
\label{lem:fi-quasiconc_weak2weak-glob}
Let $\phi$ be a quasi-concave function. If $\phi^q$ is concave or quasi-concave for some $q>p\ge1$, then $M_p: M^*(X) \to M^*(X)$ is bounded.
\end{lem}
\begin{proof}[Sketch of proof]
Analogously as in the proof of Lemma~\ref{lem:fi-quasiconc_weak2weak}, we show that \eqref{eq:phi-p-avg_bdd-glob} holds true, i.e., $\sup_{t>0} \phi(t)^p \fint_0^t \frac{ds}{\phi(s)^p} < \infty$ since the expression can be estimated from above by $q/(q-p)$ for all $t\in\Rbb^+$. By Proposition~\ref{pro:Mp-weak2weak-glob}, we can conclude that $M_p: M^*(X) \to M^*(X)$ is bounded.
\end{proof}
%
%
In order to show that \eqref{eq:phi-p-avg_bdd} holds, we may, roughly speaking, measure the ``modulus of quasi-concavity'' of the fundamental function. This gives us a slightly finer condition that generalizes the one established in the previous lemma.
\begin{lem}
\label{lem:m_phi-in-Lp}
Given a quasi-concave function $\phi$, let $m_\phi (s) = \sup_{0<t<1} \phi(t) / \phi(st)$ for $s\in(0,1)$. If $m_\phi \in L^p(0,1)$, then \eqref{eq:phi-p-avg_bdd} is satisfied.
\end{lem}
%
%
\begin{proof}
For $0<t<1$, a change of variables yields
\[
  \phi(t)^p \fint_0^t \frac{ds}{\phi(s)^p} = \phi(t)^p \int_0^1 \frac{ds}{\phi(st)^p} \le \|m_\phi\|_{L^p(0,1)}^p < \infty.
  \qedhere
\]
\end{proof}
%
%
Diverse properties of \ri spaces can be captured and described by various indices. The Boyd indices and the fundamental (Zippin) indices belong to the best studied characteristics. We will see that the upper Boyd index of an \ri space $X$ can be used to determine whether $M_p: X\to X$ is bounded, whereas the upper fundamental index determines whether $M_p: M^*(X) \to M^*(X)$ is bounded.
%
%
\begin{df}
\label{df:indices}
Let $X$ be an \ri space with fundamental function $\phi$. For $s>0$, we define the \emph{dilation operator} $E_s$ acting on $\Mcal(\Rbb^+, \lambda^1)$ as $E_sf(t) = f(st)$, $t>0$.

For $s>0$, let us define
\[
  h_X(s) = \sup_{0\neq f\in\reps X}\frac{\|E_{1/s} f\|_{\reps X}}{\|f\|_{\reps X}} \quad \mbox{and}\quad k_X(s) = \sup_{t>0} \frac{\phi(st)}{\phi(t)} =\sup_{t>0}\frac{\|E_{1/s} \chi_{(0, t)}\|_{\reps X}}{\|\chi_{(0, t)}\|_{\reps X}}\,.
\]  
Then, we define the \emph{upper Boyd index} $\itoverline{\alpha}_X$ of $X$, and the \emph{upper fundamental index} $\itoverline{\beta}_X$ of $X$ (also called the \emph{upper Zippin index}) by
\[
  \itoverline{\alpha}_X = \inf_{s>1} \frac{\log h_X(s)}{\log s} \quad \mbox{and} \quad
  \itoverline{\beta}_X = \inf_{s>1} \frac{\log k_X(s)}{\log s}\,.
\]
\end{df}
%
%
\begin{rem}
It shown in Bennett and Sharpley~\cite[Section III.5]{BenSha} that $1\le k_X(s) \le h_X(s) \le s$ for $s\ge 1$. Hence, the indices satisfy $\itoverline{\beta}_X \le \itoverline{\alpha}_X$ and both lie in $[0,1]$. Moreover, the infima in the definition of the indices can be determined as limits as $s\to \infty$. For more details, see also Boyd~\cite{Boy}, Zippin~\cite{Zip}, or Maligranda~\cite{Mali84}.
\end{rem}
%
%
\begin{lem}
\label{lem:beta_m}
If $\itoverline{\beta}_X < 1/p$, then the function $m_\phi$ of Lemma~\ref{lem:m_phi-in-Lp} lies in $L^p(0,1)$ and hence $M_p: M^*(X) \to M^*(X)_\fm$ is bounded.
\end{lem}
%
%
\begin{proof}
There exist $q \in (p, 1/\itoverline{\beta}_X)$ and $s_0 > 1$ such that $k_X(s) \le s^{1/q}$ for all $s>s_0$. We can estimate
\[
  m_\phi(s) \le k_X(s^{-1}) \le
  \begin{cases}
    s^{-1/q}& \mbox{for }s\in(0, s_0^{-1}), \\
    s^{-1}& \mbox{for }s\in[s_0^{-1}, 1).
  \end{cases}
\]
Consequently, $m_\phi \in L^p(0,1)$. The boundedness of $M_p$ was established in Proposition~\ref{pro:Mp-weak2weak} in view of Lemma~\ref{lem:m_phi-in-Lp}.
\end{proof}
In fact, the inequality for $\itoverline{\beta}_X$ in Lemma~\ref{lem:beta_m} leads to a stronger result.
\begin{lem}
\label{lem:beta_m_glob}
If $\itoverline{\beta}_X < 1/p$, then $M_p: M^*(X) \to M^*(X)$ is bounded.
\end{lem}
\begin{proof}[Sketch of proof]
The result can be obtained similarly as in Lemma~\ref{lem:beta_m}, using Proposition~\ref{pro:Mp-weak2weak-glob} and a simple modification of Lemma~\ref{lem:m_phi-in-Lp}. If we define $\widetilde{m}_\phi$ as a global version of $m_\phi$, i.e., $\widetilde{m}_\phi(s) = \sup_{t>0} \phi(t)/\phi(st)$, then $\widetilde{m}_\phi(s) = k_X(s^{-1}) \in L^p(0,1)$. That provides us with inequality \eqref{eq:phi-p-avg_bdd-glob}, which in turn is equivalent to the boundedness of $M_p: M^*(X) \to M^*(X)$.
\end{proof}
%
%
\begin{pro}
\label{pro:alpha_M-bdd}
If $\itoverline{\alpha}_X < 1/p$, then $M_p: X \to X$ is bounded. On the other hand, if $\itoverline{\alpha}_X > 1/p$, then $M_p$ is not a bounded mapping from $X$ to $X$.
\end{pro}
%
%
\begin{proof}
Let $u\in X$. Using the embedding $L^{p,1}\emb L^p$, which follows by Lemma~\ref{lem:lorentz-embedding} since $ L^{p,1} = \Lambda^1(L^p)$ and $L^p = L^{p,p} = \Lambda^p(L^p)$, we may estimate for $t\in\Rbb^+$ that
\begin{align*}
  M_pu^*(t) &= \biggl( \fint_0^t u^*(s)^p\,ds \biggr)^{1/p} = t^{-1/p} \| u^* \chi_{(0,t)}\|_{L^p} \\ &
  \lesssim t^{-1/p} \| u^* \chi_{(0,t)}\|_{L^{p,1}} \approx t^{-1/p} \int_0^t u^*(s) s^{1/p}\,\frac{ds}{s} \eqcolon P_{1/p}u^*(t).
\end{align*}
Similarly, the embedding $L^{p} \emb L^{p,\infty} = M^*(L^p)$ yields the converse estimate
\begin{align*}
  M_p u^*(t) & = \frac{\|u^* \chi_{(0,t)}\|_{L^p}}{t^{1/p}} \gtrsim \frac{\|u^* \chi_{(0,t)}\|_{L^{p,\infty}}}{t^{1/p}} = \frac{\sup_{0<s<t} u^*(s) s^{1/p}}{t^{1/q}} t^{1/q - 1/p}\\
  & \approx \frac{\sup_{0<s<t} u^*(s) s^{1/p}}{t^{1/q}} \int_0^t s^{1/q-1/p} \frac{ds}{s} \ge t^{-1/q} \int_0^t u^*(s) s^{1/q} \frac{ds}{s} = P_{1/q} u^*(t)
\end{align*}
for arbitrary $q<p$.

According to~\cite[Theorem III.5.15]{BenSha}, the Hardy-type operator $P_{a}$ is a bounded mapping from $\reps X$ to $\reps X$ if and only if $\itoverline{\alpha}_X < a$. Suppose now that $\itoverline{\alpha}_X < 1/p$. Applying the Herz--Riesz inequality, we obtain that
\[
  \|M_pu\|_X = \|(M_pu)^*\|_{\reps X} \approx \|M_pu^*\|_{\reps X} \lesssim \|P_{1/p}u^*\|_{\reps X} \lesssim \|u^*\|_{\reps X} = \|u\|_X\,.
\]
If instead $\itoverline{\alpha}_X > 1/p$. Then, there is $q<p$ such that $\itoverline{\alpha}_X > 1/q > 1/p$. The Herz--Riesz inequality yields that
\[
  \sup_{\|u\|_X\le1} \|M_pu\|_{X} = \sup_{\|u\|_X\le1} \|(M_pu)^*\|_{\reps X} \approx \sup_{\|u^*\|_{\reps X}\le1} \|M_pu^*\|_{\reps X} \gtrsim \sup_{\|u^*\|_{\reps X}\le1} \|P_{1/q}u^*\|_{\reps X} = \infty.
\qedhere
\]
\end{proof}
%
%
\begin{rem}
Shimogaki~\cite{Shi} and Montgomery-Smith~\cite{MonSmi} have given examples of \ri spaces with $\itoverline{\beta}_X < \itoverline{\alpha}_X$. If we choose $p$ such that $\itoverline{\beta}_X < 1/p < \itoverline{\alpha}_X$, then $M_p: X \not\to X$, but $M_p: M^*(X) \to M^*(X)$ is bounded. Then, $M_p: X \to M^*(X)$ is bounded, which is a key hypothesis in Lemma~\ref{lem:key-Mp-wbdd}.

If $\itoverline{\alpha}_X = \itoverline{\beta}_X = 1/p$, then we cannot draw any satisfactory conclusions about the boundedness of $M_p$. For example, let $X=L^p$. Then, $M^*(X) = L^{p,\infty}$, whence neither $M_p: X \to X$ nor $M_p: M^*(X) \to M^*(X)$ is bounded. On the other hand, $M_p: X \to M^*(X)$ is bounded. If we instead consider $X=L^{p,q}$ for some $q>p$, then $L^p\subsetneq X \nsubset L^p_\loc$ and we can find a function $u\in X$ such that $M_pu \equiv \infty$.
\end{rem}
%
%
%
%
\section{Main results}
\label{sec:LipDensSpec}
The main theorem for $p$-Poincar\'e spaces as stated in its general form (Theorem~\ref{thm:general_Lip-dens}) depends on a weak type estimate for the maximal operator $M_p$, which may be deemed rather obscure. Using the results of Section~\ref{sec:weaktype}, we can replace this estimate by somewhat more tangible hypotheses. These will also allow us to find examples of base function spaces, for which the Newtonian functions can be approximated by Lipschitz continuous functions.

\begin{thm}
\label{thm:Mp-bdd_Lip-dens}
Assume that $\Pcal$ is a $p$-Poincar\'e space for some $p\in [1, \infty)$ and that $X \subset L^p_\fm$ is a quasi-Banach function lattice with absolutely continuous norm. Suppose that $M_p: X \to X_\fm$ is bounded. Then, the set of Lipschitz functions is dense in $\NX$.
\end{thm}
%
%
\begin{proof}
By Lemma~\ref{lem:key_est-M_bdd}, we have that the boundedness of the maximal operator implies that $\Bigl\| \sigma \chi_{\suplevp{p}{f}{\sigma}} \Bigr\|_X \to 0$ as $\sigma \to \infty$ whenever $f\in X$, where $\suplevp{p}{f}{\sigma}$ is the superlevel set of $M_pf$ with level $\sigma$. The conclusion then follows from Theorem~\ref{thm:general_Lip-dens}.
\end{proof}
%
%
\begin{exa}
\label{exa:Lip-dens-LqcapLs}
Let $\Pcal$ be a $p$-Poincar\'e space for some $p\in [1, \infty)$. Let $X=L^q \cap L^s$, with a (quasi)norm given by $\|u\|_X = \max\{\|u\|_{L^q}, \|u\|_{L^s}\}$, where $0<q\le p<s<\infty$. We shall show that Lipschitz functions are dense in $\NX$. Note that if $\meas{\Pcal} = \infty$, then $X\subsetneq L^s$. If in addition $q<1$, then $X$ is not normable.

Both spaces $L^q$ and $L^s$ have absolutely continuous (quasi)norms. Thus, so has $X$.  We also need to show that $M_p: X \to X_\fm$ is bounded. Let $u\in X$ and let $E\subset \Pcal$ be of finite measure. The H\"older inequality implies that
\begin{align*}
  \|(M_p u) \chi_E\|_X & = \max\{\|(M_p u) \chi_E\|_{L^q}, \|(M_p u) \chi_E\|_{L^s}\} \\
	& \le c_{\meas{E}} \|(M_p u) \chi_E\|_{L^s} 
	\le c_{\meas{E}} \|M_p u\|_{L^s} \le \tilde{c}_{\meas{E}} \|u\|_{L^s} \le \tilde{c}_{\meas{E}} \|u\|_X.
\end{align*}
The density result now follows from Theorem~\ref{thm:Mp-bdd_Lip-dens}. It is also worth noting that if $\meas{\Pcal} = \infty$, then $M_p u \notin X$ (unless $u=0$ a.e.), but merely $M_p u \in L^s \cap L^{p,\infty}$.

If $q\ge 1$, then it is possible to draw the same conclusion on density of Lipschitz functions in $\NX$ using Example~\ref{exa:Mp-vs-Mploc} and Theorem~\ref{thm:Mp-wbdd-dtl_Lip-dens}~\ref{it:Lip-dens_X=Mploc} below.
\end{exa}
%
%
If the base function space is in fact an \ri space, then we obtain the required weak type estimate for the maximal operator $M_p$ even when $M_p$ is merely weakly bounded (on sets of finite measure).
%
%
\begin{thm}
\label{thm:Mp-wbdd_Lip-dens}
Assume that $\Pcal$ is a $p$-Poincar\'e space for some $p\in [1, \infty)$ and that $X$ is an \ri space with absolutely continuous norm. Suppose that $M_p: X \to \wMX_\fm$ is bounded. Then, the set of Lipschitz functions is dense in $\NX$.
\end{thm}
%
%
\begin{proof}
By Lemma~\ref{lem:key-Mp-wbdd}, we have that the boundedness of the maximal operator implies that $\Bigl\| \sigma \chi_{\suplevp{p}{f}{\sigma}} \Bigr\|_X \to 0$ whenever $f\in X$ as $\sigma \to \infty$. The conclusion then follows from Theorem~\ref{thm:general_Lip-dens}.
\end{proof}
%
%
As a special case, we obtain that if $\Pcal$ supports a $1$-Poincar\'e inequality, then the Lipschitz truncations are dense in all Newtonian spaces based on arbitrary \ri spaces with absolutely continuous norm.
%
%
\begin{thm}
\label{thm:1-PI_Lip-dens}
Assume that $\Pcal$ is a $1$-Poincar\'e space and that $X$ is an \ri space with absolutely continuous norm. Then, the set of Lipschitz functions is dense in $\NX$.
\end{thm}
%
%
\begin{proof}
By Proposition~\ref{pro:Mp-wbdd}, the maximal operator $M_1: X \to M^*(X)$ is bounded whenever $X$ is an \ri space. The conclusion then follows from Theorem~\ref{thm:Mp-wbdd_Lip-dens}.
\end{proof}
%
%
The absolute continuity of the norm is crucial for the density results. It is the only hypothesis that is violated in the following example, where we find a Newtonian function that cannot be approximated by (locally) Lipschitz functions.
%
%
\begin{exa}
Locally Lipschitz functions are not dense in $\NX$ in the setting of Example~\ref{exa:trunc_not_dense}. There, $X = M_\phi = M_\phi^*$ over $\Rbb^n$ and we considered a compactly supported radially decreasing function $u(x) = (f(|x|) - f(1))^+$, where $f(t) = t/\phi(t^n)$ for $t>0$ and $f(0) = f(0\limplus) = \infty$ due to the assumed properties of the fundamental function $\phi$. We also obtained that $g(x) = |f'(|x|)| \chi_{B(0,1)}(x)$, $x\in \Rbb^n$, was a minimal ($X$-weak) upper gradient of $u$. Moreover, we estimated that $|f'(t)| \approx 1/\phi(t^n)$.

Let now $v \in \NX$ be a locally Lipschitz function. The restriction $v|_{\overline{B(0,1)}}$ is a bounded $L$-Lipschitz function for some $L>0$. Let $h\in X$ be an upper gradient of $u-v$. Then, $h(x) \ge g(x) - L$ for a.e.\@ $x\in B(0,1)$. Hence
\begin{align*}
  \|u-v\|_\NX & = \|u-v\|_{X} + \|h\|_X \ge \|h \chi_{B(0,1)}\|_X \gtrsim \|(g(x) - L)^+\|_{M_\phi^*} \\
  & = \sup_{0<t<1} (|f'(t)| - L)^+ \phi(\omega_n t^n) \gtrsim \sup_{0<t<r_L} |f'(t)| \phi(t^n) \approx 1,
\end{align*}
where $\omega_n$ is the measure of the unit ball and $r_L = \inf\{r>0: |f'(r)| \le 2L\} > 0$. Therefore, $u\in \NX$ cannot be approximated by locally Lipschitz functions.
\end{exa}
%
%
In Section~\ref{sec:weaktype}, we elaborated various conditions that guarantee weak boundedness of the maximal operator $M_p$ on sets of finite measure. We may therefore concretize the assumptions of  Theorem~\ref{thm:Mp-wbdd_Lip-dens}. It becomes apparent that the well-known results on density of Lipschitz functions in $N^{1,p}$ on doubling $p$-Poincar\'e spaces, cf.\@ Shanmugalingam~\cite[Theorem 4.1]{Sha}, are recovered by our approach.
%
%
\begin{thm}
\label{thm:Mp-wbdd-dtl_Lip-dens}
Assume that $\Pcal$ is a $p$-Poincar\'e space for some $p\in [1, \infty)$ and that $X$ is an \ri space with absolutely continuous norm and fundamental function $\phi$. Suppose that any of the following conditions is satisfied:
\begin{enumerate}
	\item \label{it:Lip-dens_X=Mploc} $X \emb M^p_\loc(X)$;
	\item \label{it:Lip-dens_X=Lambdap} $X \emb \Lambda^p_{\psi,\fm}$, where $\psi$ is defined by \eqref{eq:quasiconv_dom};
	\item \label{it:Lip-dens_Lp=X} $L^{p, 1}(\Pcal) \emb X_\fm$ and $X \emb L^p_\fm(\Pcal)$;
	\item \label{it:Lip-dens_wbdd} $\phi(t)^p \fint_0^t \phi(s)^{-p}\,ds$ is bounded on $(0, \delta)$ for some $\delta > 0$;
	\item \label{it:Lip-dens_conc} $\phi^q$ is concave on $[0, \delta)$ for some $q>p$ and some $\delta>0$;
	\item \label{it:Lip-dens_quasiconc} $\phi(t)^q/t$ is decreasing for $t\in(0, \delta)$ for some $q>p$ and some $\delta>0$;
	\item \label{it:Lip-dens_m-Lp} $m_\phi \in L^p(0,1)$, where $m_\phi(s) = \sup_{0<t<1} \phi(t)/\phi(st)$;
	\item \label{it:Lip-dens_beta} the upper fundamental index $\itoverline{\beta}_X < 1/p$ (see Definition~\ref{df:indices});
	\item \label{it:Lip-dens_alpha} the upper Boyd index $\itoverline{\alpha}_X < 1/p$ (see Definition~\ref{df:indices}) .
\end{enumerate}
Then, the set of Lipschitz functions is dense in $\NX$.
\end{thm}
%
%
\begin{proof}
If $X \emb M^p_\loc(X)$, then $M_p: X\to M^*(X)_\fm$ is bounded by Proposition~\ref{pro:Mp-wbdd}. The conclusion in the case~\ref{it:Lip-dens_X=Mploc} then follows from Theorem~\ref{thm:Mp-wbdd_Lip-dens}.

If~\ref{it:Lip-dens_X=Lambdap} holds, then so does~\ref{it:Lip-dens_X=Mploc} by Lemma~\ref{lem:Lambdap_MpX}.

If~\ref{it:Lip-dens_Lp=X} holds, then $\phi_X(t) \approx t^{1/p}$ for $t$ near zero. Hence, $M^*(X)_\fm = L^{p,\infty}_\fm$ with equivalent quasi-seminorms. Let $E \subset \Pcal$ be of finite measure. In view of the Herz--Riesz inequality (Corollary~\ref{cor:herz}), we have that
\begin{multline*}
  \|(M_pu) \chi_E \|_{M^*(X)}  \le c_E \|(M_pu) \chi_E \|_{L^{p,\infty}} \le c_E \|(M_pu)^* \chi_{(0, \meas{E})} \|_{L^{p,\infty}} \\
   \approx c_E \|M_p u^* \chi_{(0, \meas{E})} \|_{L^{p,\infty}} =  c_E \|u^* \chi_{(0, \meas{E})} \|_{L^p} = c_E \|u \chi_G\|_{L^p} \le c_E' \|u\|_X,
\end{multline*}
where $G \in \suplevr{u}{\meas{E}}$ and $\suplevr{u}{\cdot}$ is the family of\fcrim{} ``superlevel sets'' defined by \eqref{eq:df-suplevr}. Thus, $M_p: X \to M^*(X)_\fm$ is bounded. Theorem~\ref{thm:Mp-wbdd_Lip-dens} now finishes the argument.

If~\ref{it:Lip-dens_wbdd} holds, then $M_p: M^*(X) \to M^*(X)_\fm$ is bounded by Proposition~\ref{pro:Mp-weak2weak}. The conclusion then follows from Theorem~\ref{thm:Mp-wbdd_Lip-dens} since $X \emb M^*(X)$.

If~\ref{it:Lip-dens_conc} or~\ref{it:Lip-dens_quasiconc} is satisfied, then so is~\ref{it:Lip-dens_wbdd} by Lemma~\ref{lem:fi-quasiconc_weak2weak}.

If~\ref{it:Lip-dens_m-Lp} holds, then so does~\ref{it:Lip-dens_wbdd} by Lemma~\ref{lem:m_phi-in-Lp}.

If~\ref{it:Lip-dens_beta} is satisfied, then so is~\ref{it:Lip-dens_m-Lp} by Lemma~\ref{lem:beta_m}.

If~\ref{it:Lip-dens_alpha} holds, then $M_p: X\to X$ is bounded by Proposition~\ref{pro:alpha_M-bdd}. Besides, $X \subset L^1_\fm$ by~\ref{df:BFL.locL1}. The desired result follows from Theorem~\ref{thm:Mp-bdd_Lip-dens}.
\end{proof}
%
%
In complete metric spaces, the conditions on the function space can be weakened.
%
%
\begin{thm}
\label{thm:Mp-wbdd-dtl_Lip-dens-comp}
Assume that $\Pcal$ is a complete $p$-Poincar\'e space for some $p\in [1, \infty)$ and that $X$ is an \ri space with absolutely continuous norm and fundamental function $\phi$. Suppose that any of the following conditions is satisfied:
\begin{enumerate}
	\item \label{it:Lip-dens_conc-comp} $\phi^p$ is concave on $[0, \delta)$ for some $\delta>0$;
	\item $\phi^p(t)/t$ is decreasing for $t\in (0, \delta)$ for some $\delta>0$;
	\item the upper fundamental index $\itoverline{\beta}_X \le 1/p$;
	\item the upper Boyd index $\itoverline{\alpha}_X \le 1/p$.
\end{enumerate}
Then, the set of Lipschitz functions is dense in $\NX$.
\end{thm}
\begin{proof}
If $p=1$, then all of the conditions are trivially satisfied (recall that $\phi$ is concave for a well-chosen equivalent norm on $X$) and the claim follows by Theorem~\ref{thm:1-PI_Lip-dens}.

Suppose instead that $p>1$. Keith and Zhong~\cite{KeiZho} have proven that $\Pcal$, being complete, supports an $r$-Poincar\'{e} inequality for some $r<p$. The claim then follows by Theorem~\ref{thm:Mp-wbdd-dtl_Lip-dens}, where $p$ and $q$ are to be replaced by $r$ and $p$, respectively.
\end{proof}
%
%
Costea and Miranda discussed the density of Lipschitz functions in Newtonian spaces based on the Lorentz $L^{p,q}$ spaces in~\cite{CosMir}. There, they showed the density whenever $1 \le q \le p < \infty$ under the assumptions that the underlying metric measure space is complete and supports an \emph{$L^{p,q}$-Poincar\'e inequality}, i.e., there are $c_\PI>0$ and $\lambda \ge 1$ such that
\[
  \fint_B |u - u_B| \le c_\PI \diam (B) \frac{\|g \chi_{\lambda B}\|_{L^{p,q}}}{\meas{\lambda B}^{1/p}}
\]
for every ball $B\subset \Pcal$, every function $u\in L^1_\loc(\Pcal)$ and every upper gradient $g$ of $u$. They gave an example showing that one cannot hope for the wanted result while considering $L^{p,\infty}$ spaces. The question of $L^{p,q}$ for $1\le p<q<\infty$ remained however open. In their proof, the Lorentz-type maximal operator
\[
  M_{p,q} u (x) = \sup_{B \ni x} \frac{\|u \chi_B\|_{L^{p,q}}}{\meas{B}^{1/p}}, \quad x\in\Pcal,
\]
and its boundedness as a mapping from $L^{p,q}$ to $L^{p, \infty}$ for $q \le p$ were used, which is exactly the place where the proof would have failed for $q>p$. The boundedness of $M_{p,q}: L^{p,q}\to L^{p, \infty}$ for $q> p$ was however left unsolved in~\cite{CosMir}. Chung, Hunt and Kurtz~\cite[pp. 119--120]{ChuHunKur} gave an example showing that $M_{p,q}$ is not bounded from $L^{p,q}$ to $L^{p, \infty}_\fm$ if $q>p$. 
Nevertheless, the following propositions give an affirmative answer to the density of Lipschitz continuous functions in $N^1L^{p,q}(\Pcal)$ even for $1<p<q<\infty$. First, we will assume a stronger Poincar\'{e} inequality.
%
%
\begin{pro}
\label{pro:cosmir-gen}
If $\fcrim\Pcal$ supports an $L^{r,s}$-Poincar\'{e} inequality for some $r,s\in[1, \infty]$, then Lipschitz functions are dense in $N^1 L^{p,q}$ whenever $p\in(r, \infty)$ and $q\in[1, \infty)$.
\end{pro}
%
%
\begin{proof}
Due to the embedding between Lorentz spaces $L^{\textit{\v{r}}} \emb L^{r,s}_\fm$, we see that $\Pcal$ is actually an $\textit{\v{r}}$-Poincar\'{e} space whenever $\textit{\v{r}} \in (r,p)$. Since $q<\infty$, the space $L^{p,q}$ has absolutely continuous norm. The fundamental function of $L^{p,q}$ satisfies $\phi(t)^p = t$, which is concave. Therefore, the condition~\ref{it:Lip-dens_conc} of Theorem~\ref{thm:Mp-wbdd-dtl_Lip-dens} is fulfilled for $X=L^{p,q}$ and the $\textit{\v{r}}$-Poincar\'{e} space $\Pcal$, whence Lipschitz functions are dense in $N^1L^{p,q}$.
\end{proof}
%
%
Finally, we are prepared to show the density result for the case $1<p<q<\infty$ in the setting of~\cite{CosMir}, cf.\@ Theorem 6.9 therein. 
%
%
\begin{pro}
\label{pro:cosmir}
Let $(\Pcal, \dd, \mu)$ be a complete metric measure space with a doubling measure. Suppose that $\Pcal$ admits an $L^{p,q}$-Poincar\'{e} inequality with $1<p<q<\infty$. Then, Lipschitz functions are dense in $N^1 L^{p,q}$.
\end{pro}
%
%
\begin{proof}
Due to the embedding between Lorentz spaces, $\Pcal$ is actually a $p$-Poincar\'{e} space. Lipschitz functions are dense in $N^1L^{p,q}$ by Theorem~\ref{thm:Mp-wbdd-dtl_Lip-dens-comp}\,\ref{it:Lip-dens_conc-comp} as $\phi_{L^{p,q}}(t)^p = t$ is concave on $\Rbb^+$.
\end{proof}
%
%
Note that the previous proposition does not discuss the case $1=p<q < \infty$ as $L^{1,q}$ are not  \ri spaces for $q>1$, but mere rearrangement-invariant quasi-Banach function lattices. It was shown in~\cite[Example~2.6]{Mal1} that the Newtonian space may be trivial then, i.e., $N^1 L^{1,q} = L^{1,q}$, even though there are many curves in the metric measure space. In this case, the density can be established using arguments similar to those in Proposition~\ref{pro:NX=X-dens} above.
%
%
%
%

\end{document}